\documentclass[reqno]{amsart}
\usepackage{url}
\usepackage{amssymb,amsthm,amsfonts,amstext}
\usepackage{amsmath}

\usepackage{mathptmx}       
\usepackage{helvet}         
\usepackage{courier}        

\usepackage{graphicx}
\usepackage{enumerate}
\usepackage[active]{srcltx}
\usepackage{mathrsfs}

\usepackage{tikz}
\usetikzlibrary{backgrounds}
\usetikzlibrary{patterns,fadings}
\usetikzlibrary{arrows,decorations.pathmorphing}
\usetikzlibrary{calc}
\definecolor{light-gray}{gray}{0.95}
\usepackage{graphicx}
\usepackage{epsfig}
\usetikzlibrary{shapes}
\usetikzlibrary{decorations}

\usetikzlibrary{decorations.pathmorphing}
\usetikzlibrary{decorations.pathreplacing}
\usetikzlibrary{decorations.shapes}
\usetikzlibrary{decorations.text}
\usetikzlibrary{decorations.markings}
\usetikzlibrary{decorations.fractals}
\usetikzlibrary{decorations.footprints}

\usepackage{color}

\newcommand\bx{{\mathbf x}}
\newcommand\by{{\mathbf y}}

\newcommand\be{{\mathbf e}}

\newcommand\bk{{\mathbf k}}

\newcommand\ve{\varepsilon}

\newtheorem{theorem}{Theorem}[section]
\newtheorem{lemma}[theorem]{Lemma}
\newtheorem{proposition}[theorem]{Proposition}
\newtheorem{corollary}[theorem]{Corollary}
\newtheorem{remark}[theorem]{Remark}
\newtheorem{definition}[theorem]{Definition}

\numberwithin{equation}{section}

\newcommand{\sfrac}[2]{{\smash{\frac{#1}{#2}}}}

\newcommand{\mc}[1]{{\mathcal #1}}
\newcommand{\mf}[1]{{\mathfrak #1}}

\newcommand{\bb}[1]{{\mathbb #1}}

\renewcommand{\>}{\rangle}

\renewcommand{\epsilon}{\varepsilon}

\newcommand{\RR}{\mathbb R}
\newcommand{\R}{\mathbb R}
\newcommand{\Z}{\mathbb Z}
\newcommand{\cF}{\mathcal F}

\newcommand{\cj}{\mathcal{J}}

\newcommand{\cv}{\mathcal{V}}

\newcommand{\C}{\mathbb{C}}

\newcommand{\NN}{\mathbb{N}}

\newcommand{\TT}{\mathbb{T}}

\newcommand{\ZZ}{\mathbb{Z}}

\newcommand{\cL}{\mathcal{L}}

\let\ve=\varepsilon

\let\ve=\varepsilon

\begin{document}

\author{C\'edric Bernardin}

\address{\noindent Universit\'e de Nice Sophia-Antipolis, Laboratoire J.A. Dieudonn\'e, UMR CNRS 7351, Parc Valrose, 06108 Nice cedex 02, France. \newline e-mail: \rm \texttt{cbernard@unice.fr}}

\author{Patr\'icia Gon\c{c}alves}

\address{\noindent Departamento de Matem\' atica, PUC-RIO, Rua Marqu\^es de S\~ao Vicente, no. 225, 22453-900, Rio de Janeiro, Rj-Brazil and
CMAT, Centro de Matem\'atica da Universidade do Minho, Campus de Gualtar, 4710-057 Braga, Portugal.\newline e-mail: \rm \texttt{patg@math.uminho.pt and patricia@mat.puc-rio.br}}

\author{Milton Jara}
\address{ Milton Jara\\ IMPA\\ Estrada Dona Castorina 110\\ Jardim Bot\^anico\\ CEP 22460-340\\ Rio de Janeiro\\ Brazil
 \newline e-mail: \rm \texttt{mjara@impa.br}}

\author{Makiko Sasada}
\address{Department of Mathematics, Keio University, 3-14-1,
Hiyoshi, Kohoku-ku, Yokohama-shi, Kanagawa, 223-8522, Japan. \newline e-mail: \rm \texttt{sasada@math.keio.ac.jp}}

\author{Marielle Simon}
\address{\noindent Departamento de Matem\' atica, PUC-RIO, Rua Marqu\^es de S\~ao Vicente, no. 225, 22453-900, Rio de Janeiro, Rj-Brazil and
UMPA ENS de Lyon, 46 all\'ee d'Italie, 69007 Lyon, France. 
 \newline e-mail: \rm \texttt{marielle.simon@mat.puc-rio.br}}

\title[]{From normal diffusion to superdiffusion of energy in the evanescent flip noise limit}

\noindent\keywords{}

\begin{abstract}
We consider a harmonic chain perturbed by an energy conserving noise depending on a parameter $\gamma$. When $\gamma$ is of order one, the energy diffuses according to the standard heat equation after a space-time diffusive scaling. On the other hand, when $\gamma=0$, the energy superdiffuses according to a $3/4$ fractional heat equation after a subdiffusive space-time scaling. In this paper, we study the existence of a crossover between these two regimes as a function of $\gamma$.
\end{abstract}

\maketitle

\section{Introduction}

Over the last few years,  superdiffusion of energy in one-dimensional Hamiltonian systems conserving momentum has attracted a lot of interest. In particular it is expected to hold for one-dimensional chains of oscillators when unpinned, and more generically for dynamical systems which preserve momentum in addition to the energy (see the review papers \cite{LLP, D}). A proof of this behavior is still unreached even if some recent progresses have been accomplished \cite{Sp}.

In order to get rigorous results in this field it has been proposed to perturb the Hamiltonian systems by a stochastic noise conserving energy and momentum.  In \cite{BBO2} the thermal conductivity is proved to be infinite for an unpinned harmonic chain of oscillators perturbed by an energy-momentum conserving noise. When the stochastic perturbations of harmonic systems do not conserve momentum or if the chain is pinned (so that momentum is not conserved), the thermal conductivity is always finite \cite{BO2, BLL} and energy diffuses. The progresses in the anharmonic case are still modest (see \cite{BG} for the case of exponential interactions).

In this work we focus on the harmonic chain of oscillators perturbed by an energy conserving noise, as considered in \cite{BerSto}. After a simple symplectic change of variables, the dynamical state of the harmonic chain is $(\omega_x)_{x\in \ZZ} \in \RR^{\ZZ}$, the energy is $\sum_{x\in\mathbb{Z}} \omega_x^2$ and Newton's equations are given by
\begin{equation}
\label{eq:dyneq}
d\omega_{x} (t) =\big(\omega_{x+1}  (t) - \omega_{x-1}(t) \big) dt, \quad x \in \ZZ.
\end{equation}
There are several ways to perturb the dynamics in order to conserve the energy. The \emph{flip} noise changes $\omega_x$ into $-\omega_x$ independently on each site $x$ at random exponential times with intensity $\gamma \geqslant 0$. Observe that the total energy is conserved by the flip noise. The \emph{exchange} noise exchanges nearest neighbor values $\omega_x$ and $\omega_y$ at random Poissonian times with intensity $\lambda \geqslant 0$. Apart from the energy, this second perturbation conserves the so-called \emph{total volume} of the chain $\sum_{x\in\mathbb{Z}} \omega_x$. The harmonic chain perturbed by the flip noise is diffusive \cite{BerPspde,Sim13} while the harmonic chain perturbed by the exchange noise is superdiffusive \cite{BerSto}. This drastic difference is due to the volume conservation which plays a role similar to the momentum conservation in chains of oscillators (see \cite{BerSto} for more explanations).

In this work we consider the harmonic chain perturbed by the flip noise with intensity $\gamma>0$ and the  exchange noise with intensity $\lambda > 0$. The two noises conserve the energy $\sum_{x\in\mathbb{Z}} \omega_x^2$  but only the latter conserves the total volume $\sum_{x\in\mathbb{Z}} \omega_x$. If $\gamma = 0$, the volume is conserved, the energy transport is superdiffusive and described by a Levy process governed by a fractional Laplacian. This has been recently proved in  \cite{BGJ}. If $\gamma > 0$, the volume is no longer conserved and one can prove that the energy transport is diffusive and described by a Brownian motion. Our aim is to study the case $\gamma \to 0$, with $\lambda$ of order $1$, and to obtain a crossover in a suitable time scale between these two very different regimes. The strength of $\gamma$ is regulated by a scaling parameter $n^{-1}$ going to $0$ and we take $\gamma \sim n^{-b}$ where $b>0$. We investigate the time scale $tn^{a}$ ($a>0$) that we have to consider in order to see some macroscopic evolution of the energy and we identify this evolution.


When $\gamma$ is sufficiently small, i.e. $b$ sufficiently large, we follow the approach of \cite{BGJ} and show that the energy superdiffuses. This can be proved for $b>1$ and the superdiffusion is described by the same Levy process as in the $\gamma=0$ case. If $\gamma$ is not sufficiently small, i.e. $b$ is too small, the techniques of \cite{BGJ} fail and we use Varadhan's approach which consists to decompose the energy current $j_{x,x+1}$ between the site $x$ and the site $x+1$ as a sum of a discrete gradient term $\nabla f_x$ and a small fluctuating term. This decomposition is known in the mathematical literature as a \emph{fluctuation-dissipation} equation and can be understood as a microscopic version of Fourier's law. In our case, some new interesting features appear with respect to the models for which normal behavior is usually observed when writing the fluctuation-dissipation equation. Indeed, the function $f_x$ involved is no longer local (see also \cite{BasOlla} for similar features).  We are able to write the fluctuation-dissipation equation for any value of $b \geqslant 0$ but we are only able to identify the limit of the energy fluctuation field for $b\in [0,2/3)$. The case $b \in [2/3 , 1]$ remains open.

\section{Model and notations}

We consider an infinite chain of harmonic oscillators at equilibrium perturbed by an energy conserving noise. The space of configurations is given by $\Omega=\R^\Z$. We say that a function $f:\Omega\to\R$ is \emph{local} if there exists a finite subset $\Lambda$ of $\Z$ such that the support of $f$ is included in $\Lambda$. Let $n\geqslant 1$ be a scaling parameter. The Liouville operator corresponding to the harmonic chain (\ref{eq:dyneq}) is given by
\begin{align*}
\mc A = \sum_{x \in \ZZ}  \left(\omega_{x+1}-\omega_{x-1}\right) \partial_{\omega_x}.
\end{align*}
The generator of the perturbed harmonic chain under investigation is given by
\begin{equation*}
{\mc L_n= \mc A+\gamma_n {\mc S}^{\text{flip}}+\lambda {\mc S}^{\text{exch}}}
\end{equation*}
where for all smooth local bounded functions $f: \Omega \to \R$  we defne
\begin{align*}
{\mc S}^{\text{flip}} f(\omega)&= \sum_{x\in \ZZ}f(\omega^x)-f(\omega),\\
{\mc S}^{\text{exch}} f(\omega) &=\sum_{x \in \ZZ} f(\omega^{x,x+1})-f(\omega).
\end{align*}
Here, the configuration $\omega^{x}$ is the configuration obtained from $\omega$ by flipping the variable $\omega_x$, i.e. $(\omega^x)_z =\omega_z$,  if $z \neq x$, and  $(\omega^x)_x=-\omega_x$. The configuration $\omega^{x,x+1}$ is obtained from $\omega$ by exchanging $\omega_x$ and $\omega_{x+1}$, i.e. $(\omega^{x,x+1})_z=\omega_z$, if $z \neq x,x+1$, and $(\omega^{x,x+1})_x=\omega_{x+1}, (\omega^{x,x+1})_{x+1}=\omega_x$. We denote by ${\mc S}_n:=\gamma_n {\mc S}^{\text{flip}}+\lambda {\mc S}^{\text{exch}}$ the total generator of the noise, where $\gamma_n,\lambda >0$ are two positive parameters which regulate the respective strengths of the noises. We assume that
\begin{equation*}
\gamma_n = \frac{c}{n^b}, \quad c,b>0.
\end{equation*}

The process with generator ${\mc L}_n$ is denoted by $\{\omega (t)\}_{t \geqslant 0} := \{ \omega_x (t) \, ; \, x \in \ZZ\}_{t \geqslant 0}$.

The Gibbs equilibrium measures of $\{ \omega (t)\}_{t \geqslant 0}$ are given by the Gaussian product probability measures
\begin{equation*}
\mu_{\beta} (d\omega) = \prod_{x \in \ZZ} \sqrt{{\frac{\beta}{2\pi}}} \exp\Big(-\frac{\beta \omega_x^2 }{2}\Big) d\omega_x,
\end{equation*} where $\beta >0$ stands for the inverse temperature.
In the following, the expectation of a function $f$ with respect to $\mu_\beta$ is denoted by $\langle f \rangle_\beta$ and the covariance between functions $f,g$ with respect to $\mu_\beta$ is denoted by  $\langle f ; g\rangle_\beta$.

We consider the dynamics starting from a Gibbs equilibrium  measure at a fixed temperature $\beta^{-1}$ in the time scale $tn^a$, $a>0$. The existence of the infinite dynamics under this initial distribution, can be proved by following for instance \cite{LLL} or \cite{ffL}. We denote by ${\mc S}(\RR)$ the Schwarz space of rapidly decreasing functions.

Let us fix a time horizon line $T>0$. We define the energy fluctuation field $\{ {\mc E}^n_t \; ; \; t \in [0,T]\}$  in the time scale $tn^a$ as the  ${\mc S}(\RR)$-valued process given by
\begin{equation*}
{\mc E}_t^n (f) = \tfrac{1}{\sqrt{n}} \sum_{x \in \ZZ} f \big(\tfrac{x}{n}\big) \left\{\omega^2_x (tn^a) -\beta^{-1} \right\}.
\end{equation*}

Similarly, the volume fluctuation field $\{ {\mc V}^n_t \; ; \; t \in [0,T]\}$ in the time scale $tn^a$ is  the  ${\mc S}(\RR)$-valued process given by
\begin{equation*}
{\mc V}_t^n (f) = \tfrac{1}{\sqrt{n}} \sum_{x \in \ZZ} f \big(\tfrac{x}{n}\big) \omega_x (tn^a).
\end{equation*}
It is not difficult to check that these two fields are well-defined almost surely.

%
%
%

\section{Statement of the results}

Given two functions $f,h \in {\mc S}(\RR)$, we look at the evolution with $t$ of the space-time correlation function of the energy fluctuation field
\begin{equation*}
\begin{split}
\sigma^n_{t} (f,h)&= \langle {\mc E}_t^n (f) \, ; \, {\mc E}_0^n (h) \rangle_\beta\\
&=\tfrac{1}{n}\sum_{x,y\in \ZZ} f\big(\tfrac{x}{n}\big)h\big(\tfrac{y}{n}\big) \left\langle \big(\omega^2_x(tn^a)-\beta^{-1} \big)\big(\omega^2_y(0)-\beta^{-1}\big)\right\rangle_\beta\\
& =\tfrac{1}{n}\sum_{y,z\in \ZZ} f\big(\tfrac{y+z}{n}\big)h\big(\tfrac{y}{n}\big) \left\langle \omega^2_z(tn^a) \big(\omega^2_0(0)-\beta^{-1}\big)\right\rangle_\beta,
\end{split}
\end{equation*}
when $n$ goes to infinity. The choice of $a>0$ fixes the time scale.

\begin{theorem}
\label{th:1}
Let us assume that $a=2-b/2$ and $b<2/3$.  Let $f,h\in{\mc S}(\RR)$ and let us fix $t>0$. Then,
\[\lim_{n\to\infty} \sigma_t^n(f,h)=\frac{1}{\beta^2\sqrt{\pi t \kappa}}\iint_{\R^2} {d}u {d}v f(u)h(v)\exp\left(-\frac{(u-v)^2}{4t\kappa}\right),\]
where \[\kappa=\begin{cases} \displaystyle \frac{1}{\sqrt{2\lambda c}} & \text{ if } b\in (0, 2/3),\\
\lambda +\displaystyle \frac{1}{\sqrt{2\lambda c}} & \text{ if } b=0.\end{cases}\]
\end{theorem}

One can also prove that the fluctuation field ${\mc E}_t^n$ converges in law to the infinite dimensional Ornstein-Uhlenbeck (OU) process ${\mc E}_t$, solution of the linear stochastic partial differential equation \[ \partial_t {\mc E}=\kappa \partial_u^2 {\mc E} \, dt + \sqrt{4\kappa\beta^{-2} } \, \partial_u B(u,t),\] where $B$ is the standard normalized space-time white noise. This extension of Theorem \ref{th:1} is standard, and we refer to \cite{KL} for more details.

For the sake of simplicity, hereafter we assume $\lambda=1$. The same computations for any $\lambda >0$ could be done, but become significantly more technical. In Section \ref{sec:bgrand} we prove the following theorem.

\begin{theorem}
\label{theo2}
Let $\{P_t\}_{t\geqslant 0}$ be the semigroup generated by the infinitesimal generator
\begin{equation} \label{frac ope}
{\bb L}:= - \frac{1}{\sqrt 2} \left((-\Delta)^{3/4}\, - \,  \nabla (- \Delta)^{1/4}\right).
\end{equation}
If $b>1$ and $a=3/2$ then
\begin{equation}
\label{eq:theo2}
\lim_{n \to + \infty} \sigma_t^n (f,h) = \frac{2}{\beta^2} \iint_{\R^2} {d}u {d}v\,  f(u)h(v) P_t(u-v) .
\end{equation}
\end{theorem}

As for Theorem \ref{th:1}, with a little more effort we could also prove that the energy fluctuation field ${\mc E}_t^n$ converges to an infinite dimensional 3/4-fractional Ornstein-Uhlenbeck (fOU) process.

The case $b\in [2/3,1]$ remains open. The conjecture is not easy to guess. One possible behavior is the following: $b=1$  would be a field interpolating the standard OU process and the fOU process, and $b\in [2/3,1)$ would correspond to the same diffusive behavior as for $b \in [0,2/3)$.    In any cases, it is easy, by scaling considerations, to see that the limiting energy field ${\mc E}(t,x \, ; \, c,b):= \lim_{n \to \infty} {\mc E}_t^n$  in the time scale $t n^a$ shall satisfy the scaling relation (in law)
\begin{equation}
{\mc E} (t,x\, ; \, c,b) = \ve^{1/2} {\mc E} ( t \ve^a, \ve x \, ; \, c \ve^{-b}, b), \quad \ve>0.
\end{equation}

The only result we are able to prove in the window $b\in[2/3 ,1]$ is that the time scale necessary to see some macroscopic evolution of the energy is at least $t n^{4/3}$. The proof is the same as in \cite{BG} and we refer the interested reader to that paper.

We turn now to the volume fluctuation field. The space-time correlation function of the volume fluctuation field in the time-scale $tn^a$ is defined for $f,h\in\mc{S}(\bb R)$ as
\begin{equation}\label{st corre of vol}
\begin{split}
\eta_t^n(f,h)&:=\tfrac{1}{n}\sum_{x,y\in \ZZ} f\big(\tfrac{x}{n}\big)h\big(\tfrac{y}{n}\big) \left\langle \omega_x(tn^a)\omega_y(0)\right\rangle_\beta\\
& = \tfrac{1}{n}\sum_{z,y\in \ZZ} f\big(\tfrac{y+z}{n}\big)h\big(\tfrac{y}{n}\big) \left\langle \omega_z(tn^a)\omega_0(0)\right\rangle_\beta.
\end{split}
\end{equation}

We obtain complete results, in the sense that, all time scales $a$ and all exponents $b$ are covered. The behavior of the volume fluctuations is of three different types: \begin{itemize}
\item \emph{Relaxation}: this means that $\eta_t^n$ converges to $\eta_t$ which is solution of \[\partial_t \eta_t(f,h)=-C\eta_t(f,h),\] for some constant $C>0$.
\item \emph{Transport}: this means that $\eta_t^n$ converges to $\eta_t$ which is solution of \[ \partial_t\eta_t(f,h)=C\eta_t(f',h),\] for some constant $C >0$.
\item \emph{Heat}: this means that $\eta_t^n$ converges to $\eta_t$ which is solution of \[ \partial_t\eta_t(f,h)=C\eta_t(f'',h),\] for some constant $C>0$.
\end{itemize}

More precisely let us state the theorem. For any parameter $b>0$, we are going to see that the hyperbolic time scale $a=1$ yields to a transport equation with speed constantly equal to 2. In other words, the limit volume fluctuation field at time $t$ is a translation of the initial one. As a consequence, in higher time scales, the fluctuation field should be redefined in order to take into account this transport phenomenon.  When the time scale $a$ satisfies $a> 1$, we redefine the  space-time correlation function of the volume fluctuation field  as

\[\widetilde\eta_t^n(f,h)=\tfrac{1}{n}\sum_{x,y\in \ZZ} f\big(\tfrac{x-2tn^a}{n}\big)h\big(\tfrac{y}{n}\big) \left\langle \omega_x(tn^a)\omega_y(0)\right\rangle_\beta\]
Let us remark that the two fluctuation fields initially coincide: $\eta_0^n=\widetilde\eta_0^n$.

\begin{theorem} \label{theo:volume} The behavior of the volume fluctuation field depends on the time scale in the following way: \begin{enumerate}[A)]

\item \textbf{Case} $b\leqslant 1$.

\begin{enumerate}[(i)]
\item If $a<b$, then $\partial_t \eta_t^n(f,h)$ vanishes, and \[\lim_{n\to\infty} \eta_t^n(f,h)= \eta_0(f,h)=\beta^{-1} \iint_{\R^2} dudv\, f(u)h(v).\]

\item If $a=b$, then
\[ \lim_{n\to\infty} \eta_t^n(f,h)=\iint_{\R^2} dudv\, f(u)h(v)P_t(u-v),\] where $\{P_t\}_{t\geqslant 0}$ is the semigroup generated by the infinitesimal generator 
\[\mathbf{1}_{b=1} \times 2\nabla - 2c {\rm Id}.\]

\item If $a>b$,  $ \lim_{n\to\infty} \eta_t^n(f,h)=0.$
\end{enumerate}

\item \textbf{Case} $1<b<2$.

 \begin{enumerate}[(i)]
\item If $a<b$ and $a\in (0,1)$, then $\partial_t \eta_t^n(f,h)$ vanishes, and \[\lim_{n\to\infty} \eta_t^n(f,h)= \eta_0(f,h)=\beta^{-1} \iint_{\R^2} dudv\, f(u)h(v).\]

\item If $a<b$ and $a\in (1,2)$, then $\partial_t \widetilde\eta_t^n(f,h)$ vanishes, and \[\lim_{n\to\infty} \widetilde\eta_t^n(f,h)= \eta_0(f,h)=\beta^{-1} \iint_{\R^2} dudv\, f(u)h(v).\]

\item If $a=1$,\[ \lim_{n\to\infty} \eta_t^n(f,h)=\iint_{\R^2} dudv\, f(u)h(v)P_t(u-v),\] where $\{P_t\}_{t\geqslant 0}$ is  the semigroup generated by the infinitesimal generator $2\nabla.$

\item If $a=b$, then $\widetilde\eta_t^n$ converges to $\eta_t$ where $\eta_t$ is the solution of
\[ \left\{\begin{aligned} \partial_t \eta_t(f,h)&=-2c\eta_t(f,h), \\
\eta_0(f,h)&=\beta^{-1} \iint_{\R^2} dudv\, f(u)h(v).\end{aligned} \right.\]

\item If $a>b$, $ \lim_{n\to\infty} \eta_t^n(f,h)=0.$
\end{enumerate}

\item \textbf{Case} $b\geqslant 2$.

 \begin{enumerate}[(i)]
\item If $a \in (0,1)$,   then $\partial_t \eta_t^n(f,h)$  vanishes, and \[\lim_{n\to\infty} \eta_t^n(f,h)= \eta_0(f,h)=\beta^{-1} \iint_{\R^2} dudv\, f(u)h(v).\]

\item If $a \in (1,2)$,   then $\partial_t \widetilde\eta_t^n(f,h)$  vanishes, and \[\lim_{n\to\infty} \widetilde\eta_t^n(f,h)= \eta_0(f,h)=\beta^{-1} \iint_{\R^2} dudv\, f(u)h(v).\]

\item If $a=1$, \[ \lim_{n\to\infty} \eta_t^n(f,h)=\iint_{\R^2} dudv\, f(u)h(v)P_t(u-v),\] where $\{P_t\}_{t\geqslant 0}$ is  the semigroup generated by the infinitesimal generator $2\nabla.$

\item If $a=2$, \[ \lim_{n\to\infty} \widetilde\eta_t^n(f,h)=\iint_{\R^2} dudv\, f(u)h(v)P_t(u-v),\] where $\{P_t\}_{t\geqslant 0}$ is  the semigroup generated by the infinitesimal generator \[\lambda \Delta- \mathbf{1}_{b=2} \times 2c {\rm Id}.\]

\item If $a>2$, $ \lim_{n\to\infty} \eta_t^n(f,h)=0.$
\end{enumerate}

\end{enumerate}

\end{theorem}

We resume in Figure \ref{fig:energy} and Figure \ref{fig:volume} all the results of this paper.

\begin{figure}
\begin{center}
\begin{tikzpicture}[scale=0.3]
\draw (0,23) node[left]{$a$};
\draw (23,0) node[below]{$b$};
\draw (10,0) node[below]{$1$};
\draw (6.666,0) node[below]{$\frac{2}{3}$};
\draw (0,0) node[left]{$0$};
\draw (0,16.666) node[left]{$\frac{5}{3}$};
\draw (0,13.333) node[left]{$\frac{4}{3}$};
\draw (0,15) node[left]{$\frac{3}{2}$};
\draw (0,20) node[left]{$2$};
\draw[-,=latex,red,ultra thick] (0,20) -- (6.666, 16.666) node[midway,above,sloped] {Heat eq.};
\draw[-,=latex,red,ultra thick] (10,15) -- (20,15) node[midway,above,sloped] {Fract. heat eq.};
\fill[light-gray] (0,0) -- (0,20) --(6.666,16.666) -- (6.666,13.333) -- (10,13.33) -- (10,15) -- (20,15) -- (20,0) -- cycle;
\draw (8,6.67) node{{\textit{No evolution}}};
\draw[blue, ultra thick] (6.666,16.666) -- (6.666,13.333);
\draw[blue, ultra thick] (6.666,13.333) -- (10,13.333);
\draw[blue, ultra thick] (10,13.333) -- (10,15);
\draw[-,=latex, dashed] (-0.1,16.666) -- (6.666,16.666);
\draw[-,=latex, dashed] (-0.1,13.333) -- (20,13.333);
\draw[-,=latex, dashed] (-0.1,15) -- (20,15);
\draw[-,=latex, dashed] (6.666,-0.1) -- (6.666,16.666);
\draw[-,=latex, dashed] (10,-0.1) -- (10,15);
\draw (8.2,14.2) node{{\color{blue}\textbf{?}}};
\draw[->,>=latex] (0,0) -- (24,0);

\draw[->,>=latex] (0,0) -- (0,24);
\end{tikzpicture}
\end{center}
\caption{Energy fluctuations}
\label{fig:energy}
\end{figure}
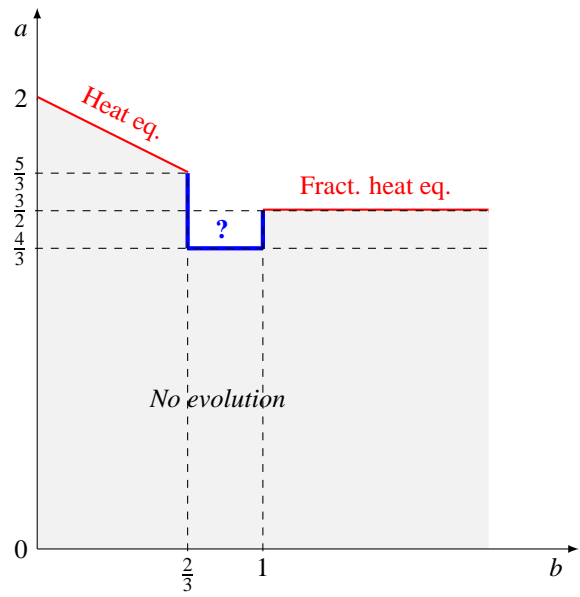

\begin{figure}
\begin{center}
\begin{tikzpicture}[scale=0.3]
\draw (0,25) node[left]{$a$};
\draw (25,0) node[below]{$b$};
\draw (10,0) node[below]{$1$};
\draw (20,0) node[below]{$2$};
\draw (0,0) node[left]{$0$};
\draw (0,10) node[left]{$1$};
\draw (0,20) node[left]{$2$};
\fill[light-gray] (0,0) -- (20,20) -- (25,20) -- (25,25) -- (0,25) -- cycle;
\fill[fill=blue, fill opacity=0.1] (0,0) -- (25,0) -- (25,20) -- (20,20) -- cycle;
\draw[-,=latex,red,ultra thick] (20,20) -- (25, 20) node[midway,below,sloped] {Heat eq.};
\draw[-,=latex,red,ultra thick] (10,10) -- (25,10) node[midway,below,sloped] {Transport eq.};
\draw[-,=latex,blue, dashed, ultra thick] (0,0) -- (10,10) node[midway,above,sloped] {Relaxation};
\draw[-,=latex,blue,dashed, ultra thick] (10,10) -- (20,20) node[midway,above,sloped] {Relaxation};
\node[circle,fill=black,inner sep=0.8mm] at (10,10) {};
\node[] at (10,10) [above] {\bf Relax. + Transport};
\node[circle,fill=black,inner sep=0.8mm] at (20,20) {};
\node[] at (20,20) [above] {\bf Relax. + Heat};
\node[] at (6,15) {\it Vanish};
\node[] at (15,5) {\it No evolution};
\node[] at (21.5,15) {\it No evolution};
\draw[-,=latex, dashed] (-0.1,10) -- (10,10);
\draw[-,=latex, dashed] (-0.1,20) -- (20,20);
\draw[-,=latex, dashed] (10,-0.1) -- (10,10);
\draw[-,=latex, dashed] (20,-0.1) -- (20,20);
\draw[->,>=latex] (0,0) -- (26,0);

\draw[->,>=latex] (0,0) -- (0,26);

\draw[thick, decorate,decoration={brace, amplitude=5pt}] (27,25) -- (27,10);

\node[] at (28,18) [right] {$\widetilde\eta_t$};

\draw[thick, decorate,decoration={brace,amplitude=5pt}] (27,10) -- (27,0);

\node[] at (28,5) [right] {$\eta_t$};

\end{tikzpicture}
\end{center}
\caption{Volume fluctuations}
\label{fig:volume}
\end{figure}
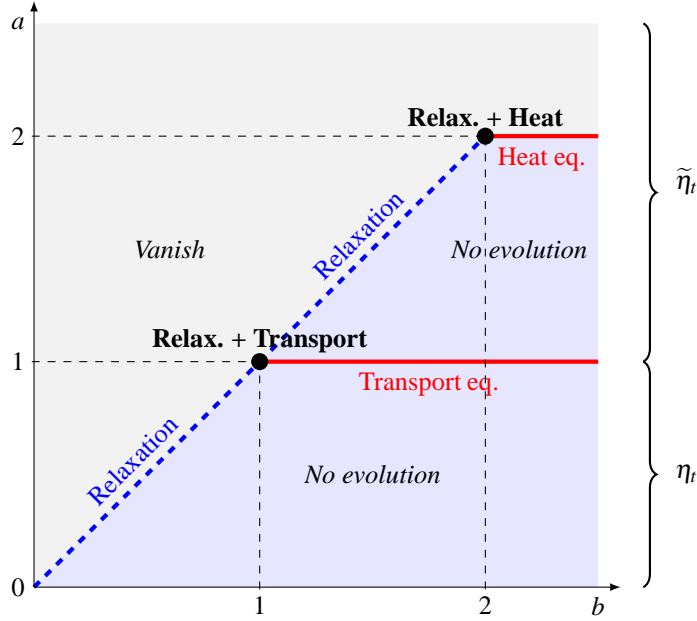

\newpage

\section{Diffusive domain for the energy}

Before  giving the proof of Theorem \ref{th:1} we recall some facts about the notion of the fluctuation-dissipation equation.  Hereafter, $\tau_x$ denotes the translated operator that acts on a function $h$ as $(\tau_x u)(\omega):=u(\tau_x \omega)$, and $\tau_x \omega$ is the configuration obtained from $\omega$ by shifting: $(\tau_x \omega)_y=\omega_{x+y}$. The microscopic current of energy is defined by the local conservation law
${\mc L}_n (\omega^2_x) =\nabla (j_{x-1,x})$, where the discrete gradient $\nabla$ is defined for a function $h:\Omega\to\R$ as \[\nabla(h)(\omega):=(\tau_{1}h)(\omega)-h(\omega).\] 
Here, the energy current is explicitly given by
\begin{equation}
\label{eq:current}
j_{x,x+1}(\omega)=j_{x,x+1}^A(\omega)+j_{x,x+1}^S(\omega)=2\omega_x \omega_{x+1}+\lambda(\omega_{x+1}^2-\omega_x^2).
\end{equation}
Let us notice that the current does not depend on $n$ (since it does not involve the intensity $\gamma_n$ of the flip noise).

If $b=0$, i.e. $\gamma_n=c$ is a constant independent of $n$, and $\lambda=0$, the volume is not conserved and we can show that the energy fluctuation field $\mc E_{t}^n$, with $a=2$, converges, as $n \to \infty$, towards an infinite dimensional Ornstein-Uhlenbeck process: the energy transport is diffusive. The argument is based on two ingredients: 1) the existence of a fluctuation-dissipation equation and 2) the Boltzmann-Gibbs principle. A {\emph{fluctuation-dissipation}} equation for the current is the decomposition of the current as a sum of a discrete gradient $\nabla f_x$ and a term in the form ${\mc L}_n (g_x)$, where $f_x$ and $g_x$ are two functions of the configuration. The \emph{Boltzmann-Gibbs principle} was first introduced by Brox and Rost \cite{BR} and, roughly speaking, states that the space-time fluctuations of any non-conserved field can be written as a linear functional of the conserved field plus a small fluctuating term. These two properties are sufficient to close the evolution equation of the energy fluctuation field (see \cite{KL} for more details).

In the case $b=\lambda=0$, a straightforward computation shows that \[ j_{x,x+1}=\nabla(\omega_x^2+\omega_{x-1}\omega_{x+1})-\cL_n (\omega_x\omega_{x+1}),\] where $\nabla$ is the discrete gradient. Thus an exact fluctuation-dissipation equation involving local functions holds{\footnote{The situation is exactly similar to the velocity-flip model investigated  in \cite{Sim13} and the results of \cite{Sim13} can be adapted to this model mutatis mutandis.}}. If $b=0$ but $\lambda>0$ then a fluctuation-dissipation equation still holds, but the functions involved in the decomposition are no longer local. However, they are exponentially localized so that this does not really matter and a normal diffusion of energy takes place (see e.g.~\cite{BasOlla}). On the other hand, if $b=\infty$ (or equivalently $c=0$) and $\lambda>0$, the situation is very different since energy superdiffuses (\cite{BGJ}). In this case, it is not clear if such a fluctuation-dissipation equation can be obtained. But if it exists it should involve very non-local functions.

 In this section we are interested in the energy transport in the limit $\gamma\to 0$. More exactly, we assume that $\gamma=\frac{c}{n^b}$ with $n \to \infty$. It turns out that if $b<\infty$, an exact fluctuation-dissipation equation still holds with exponentially localized functions. Nevertheless the length scale of the localization is typically of order $\gamma_n^{-1}$ which goes to $\infty$, as $n \to \infty$. The ``standard technique" \cite{KL} to obtain convergence of the fluctuation field to an Ornstein-Uhlenbeck process has to be adapted with care and, in fact, we are only able to do it for $b<2/3$.

%
%
%
%
%
%
%
%

\subsection{The fluctuation-dissipation equation}

\paragraph{\sc Strategy} In the sequel, we erase the dependence on the parameters $n,b$ and $c$ whenever no confusion arises. We consider a function $u$ in the form
\begin{equation}
\label{eq:uu}
u = \sum_{x\in\mathbb{Z}} \sum_{k\geqslant1} \rho_k(x) \omega_x \omega_{x+k},
\end{equation}
where $\{\rho_k(x)\, ; \, x \in \Z, k \geqslant 1\}$ is a real sequence that satisfies the condition
\begin{equation}
  \label{eq:normalization_condition_u}
  \sum_{x\in\mathbb{Z}} \sum_{k\geqslant1} |\rho_k(x)|^2 < +\infty,
\end{equation}so that  $u$ is a function in $\mathbf L^2(\mu_\beta)$. Observe first that ${\mc A}u$ is a sum of gradient terms. Indeed, we have
\begin{equation}\label{eqAu}
{\mc A}u =  \sum_{x\in\mathbb{Z}} \sum_{k\geqslant1} \rho_k(x)\,  \nabla \left[ \omega_{x-1} \omega_{x+k} + \omega_x \omega_{x+k-1}\right].
\end{equation}
 Our aim is now to solve the equation
\begin{equation} \label{eq:aim}
(\gamma_n{\mc S}^{\rm flip} + \lambda {\mc S}^{\rm exch} ) u = 2 \omega_0 \omega_1=j_{0,1}^A (\omega).
\end{equation}
Then it will follow that
\begin{equation}
\label{eq:FD}
{\mc L}_n u - v = j_{0,1}^A
\end{equation}
where $v={\mc A} u$ given in \eqref{eqAu}
is a sum of gradient functions. One can easily check that $v$ and ${\mc L}_n u$ are also in the space $\mathbf{L}^2(\mu_\beta)$.

\paragraph{\sc Resolution} Straightforward computations show  that
\begin{equation*}
(\gamma_n{\mc S}^{\rm flip} + \lambda {\mc S}^{\rm exch} ) u = \sum_{x \in \ZZ} \sum_{k\geqslant 1} F_{k} (x) \omega_x \omega_{x+k},
\end{equation*}
with, for $x\in\ZZ$,
\begin{equation*}
\begin{split}
&F_{1} (x) = -2 (2\gamma_n +\lambda) \rho_1 (x) + \lambda (\rho_2 (x) + \rho_2 (x-1)),\\
&F_{k} (x) = -4 (\gamma_n + \lambda) \rho_{k} (x)
+ \lambda \Big(\rho_{k-1} (x) + \rho_{k-1} (x+1) + \rho_{k+1} (x) + \rho_{k+1} (x-1)\Big), \quad k\geqslant 2.
\end{split}
\end{equation*}
For the sake of clarity, we forget  that the coefficients $F_k(x)$ should depend on $n$. Identifying the coefficients in front of the different terms, it follows that \eqref{eq:aim} will hold if, for all $x \in \ZZ$,
\begin{equation}
\label{eq:64}
F_{k} (x) = 2 {\bf 1}_{\{k=1, x=0\}}.
\end{equation}
In \eqref{fou tranf}  we introduce the Fourier transform $\widehat{h} \in \mathbf{L}^2(\TT)$
of a given function $h \in \ell^2(\ZZ,\RR)$.
%
Then, the condition \eqref{eq:64} can be equivalently reformulated for $\theta \in \TT$ as
\begin{equation}
\label{eq:65}
\begin{cases}
-2(2\gamma_n+\lambda) \widehat{\rho}_1 (\theta) + \lambda (1 +e^{2 i \pi \theta})
\widehat{\rho}_2(\theta) = 2,\cr
-4 (\gamma_n+\lambda) \widehat{\rho}_{k}(\theta) + \lambda (1 +e^{-2 i \pi \theta})
\widehat{\rho}_{k-1}(\theta) + \lambda (1 +e^{2 i \pi \theta})
\widehat{\rho}_{k+1}(\theta) = 0, \quad k \geqslant 2.
\end{cases}
\end{equation}
By Parseval's identity, condition \eqref{eq:normalization_condition_u} is equivalent to
\[
\sum_{k \geqslant 1} \int_{\TT} | \widehat{\rho}_k(\theta)|^2 \, d\theta < +\infty.
\]
Then, one can easily show that \eqref{eq:65} and the above integrability condition leads to
\[\widehat{\rho}_k(\theta)= \widehat{\rho}_1(\theta)(X(\theta))^{k-1},\]
with
\[
X(\theta) = \frac{2}{1+e^{2i\pi \theta}} \left\{ 1+\tfrac{\gamma_n}{\lambda} -\sqrt{(1+\tfrac{\gamma_n}{\lambda})^2 -\cos^2 (\pi \theta)}\right\}
\]
and
\begin{equation*}
{\widehat \rho}_1 (\theta) = -\cfrac{1}{\gamma_n +\lambda\sqrt{(1+\tfrac{\gamma_n}{\lambda})^2 -\cos^2 (\pi \theta)}}.
\end{equation*}
\paragraph{\sc Sharp estimates} In the following, we will need sharp estimates on $\widehat \rho_1$ and $X$, precisely:
\begin{lemma} \label{lem:estimates} For $\theta \in \TT$ and for sufficiently large $n$, we have
\[
| X(\theta) |  \leqslant \frac{|\cos(\pi \theta)|}{1+\sqrt{\frac{\gamma_n}{\lambda}}}\quad \text{ and } \quad
|\widehat{\rho}_1(\theta) |  \leqslant \frac{1}{\lambda\sqrt{\frac{\gamma_n}{\lambda}+\sin^2(\pi\theta)}}.
\]
\end{lemma}

\begin{proof}
We only prove the first estimate, since the second one is straightforward.
Let us define $C:=\gamma_n/\lambda$. Then \[ |X(\theta)|=\frac{1}{|\cos(\pi\theta)|}\left\{ 1+C -\sqrt{(1+C)^2 -\cos^2 (\pi \theta)}\right\}.\]
For $n$ large enough, we have $C<1$, and then $(1+C)^2 \leqslant (1+C)(1+\sqrt{C})$. It follows that \[
1-\frac{\cos^2(\pi\theta)}{(1+C)(1+\sqrt C)}\leqslant{1-\frac{\cos^2(\pi\theta)}{(1+C)^2}} \leqslant\sqrt{1-\frac{\cos^2(\pi\theta)}{(1+C)^2}},\]
and we get the result straightforwardly.
\end{proof}

\subsection{Strategy of the proof of Theorem \ref{th:1}}

We have
\begin{equation*}
\begin{split}
\sigma_t^n (f,h)& =\sigma_0^n (f,h) -n^{a-2} \int_0^t \sum_{y,z\in\mathbb{Z}} (\nabla_n f) \big( \tfrac{y+z}{n} \big) h\big( \tfrac{y}{n}\big) \Big\langle j_{z,z+1} (sn^a) \, , \, \omega^2_0 (0) -\beta^{-1} \Big\rangle_{\beta} ds,
\end{split}
\end{equation*}
where $\nabla_n$ is the discretization of the derivative w.r.t. the lattice $n^{-1} \ZZ$, that is,
$ (\nabla_n f) \big( \tfrac{y}{n} \big) = n \left[ f \big( \tfrac{y+1}{n} \big)- f \big( \tfrac{y}{n} \big)\right].$
The discretization $\Delta_n$ of the Laplacian is defined in a similar way. Let us recall that the current writes $j_{z,z+1} = j^A_{z,z+1} +j^S_{z,z+1}$.

If $a<2$, it is easy to see that the contribution coming from $j^S_{z,z+1}$ vanishes, as $n \to \infty$, since a second integration by parts can be performed and we are in a subdiffusive time scale. If $a=2$ and $b=0$, the symmetric part of the current gives a non trivial contribution. More precisely, after an integration by parts we will get the term
\[ \int_0^t \sigma_s^n(\lambda f'',h) \, ds,\] and therefore the coefficient $\lambda$ will appear in the thermal conductivity.

We now assume that $a<2$. By using the fluctuation-dissipation equation \eqref{eq:FD} for the contribution coming from $j_{0,1}^A$, we obtain that
\begin{equation*}
\begin{split}
\sigma_t^n (f,h)&=\sigma_0^n (f,h) +n^{a-2} \int_0^t \sum_{y,z\in\mathbb{Z}} (\nabla_n f) \big( \tfrac{y+z}{n} \big) h\big( \tfrac{y}{n}\big) \Big\langle \;  (\tau_z v )(sn^a) \, , \, \omega^2_0 (0) -\beta^{-1} \; \Big\rangle_{\beta} ds\\
&\qquad -n^{a-2} \int_0^t \sum_{y,z\in\mathbb{Z}} (\nabla_n f) \big( \tfrac{y+z}{n} \big) h\big( \tfrac{y}{n}\big) \Big\langle \; ({\mc L}_n \tau_z u) (sn^a) \, , \, \omega^2_0 (0) -\beta^{-1}\;  \Big\rangle_{\beta} ds \; + \; o_n (1)\\
&=: \sigma_0^n (f,h) + V_t^n (f,h) -U_t^n (f,h) +o_n (1).
\end{split}
\end{equation*}
Let us now focus on the term $V_t^n(f,h)$. The function $v$ can be rewritten as
\begin{equation*}
v= \sum_{x\in\mathbb{Z}} \rho_1 (x) \nabla (\omega_x^2) + \psi
\end{equation*}
where
\begin{equation}\label{theta}
\psi=\sum_{x\in\mathbb{Z}}\sum_{k\geqslant 2}\rho_k(x)\nabla \Big\{\omega_{x-1}\omega_{x+k}+\omega_{x}\omega_{x+k-1}\Big\}+\sum_{x\in\mathbb{Z}}\rho_1(x)\Big\{{\omega_{x}\omega_{x+2}-\omega_{x-1}\omega_{x+1}}\Big\}.
\end{equation}
Then, accordingly to this decomposition, we write the term $V_t^n (f,h)$ as the sum of two terms
\begin{equation*}
V_t^n (f,h) = K_t^n (f,h) + \Psi_t^n (f,h).
\end{equation*}
It turns out that
\begin{equation*}
K_t^n (f,h)= - n^{a-1} \int_0^t \sigma_s^n (F , h) ds,
\end{equation*}
where the function $F$ is defined on $\tfrac{1}{n}\Z$ by
\begin{equation*}
\begin{split}
F \big( \tfrac{w}{n} \big) &= \sum_{z \in \ZZ} \rho_1 (z) \left[ (\nabla_n f) \big( \tfrac{w-z}{n}\big) -  (\nabla_n f) \big( \tfrac{w-z-1}{n}\big) \right]\\
&= \tfrac{1}{n} \sum_{z \in \ZZ} \rho_1 (z) (\Delta_n f) \big( \tfrac{w-z}{n}\big), \qquad w\in\Z.
\end{split}
\end{equation*}
In the sequel we prove the following convergences:\begin{enumerate}[(i)]
\item If $b<1$, then
$\lim_{n \to + \infty} |U_t^n (f,h)| =0.$
\item \label{eq:ii} If $a=2-b/2$ and $b<2$, then
$ \lim_{n\to\infty} \left\vert K_t^n(f,h)-\int_0^t \sigma_s^n(\kappa f'',h) \, ds\right\vert =0.$
\item If $b<2/3$, then $ \lim_{n\to\infty} |\Psi_t^n(f,h)|=0.$

\end{enumerate}

One can easily check that these three points imply Theorem \ref{th:1}. Besides, we shall see in the proof of \eqref{eq:ii} that the case $a<2-b/2$ corresponds to $\vert V_t^n(f,h)\vert \to 0$, as $n\to\infty$. In other words, there is no evolution up to the time scale $a=2-b/2$.

\begin{remark}
If $b>2$ and $a=3/2$, we can adapt the argument in the proof of \eqref{eq:ii} and show that the limit results in  a constant times the $3/4$-fractional Laplacian of $f$ (instead of a constant times the second derivative of $f$). However, this is not sufficient to prove that for $b>2$ the limit of the energy fluctuation field is given by a fractional heat equation, because we do not know how to control the other terms ($U^n_t$ and $\Psi_t^n$) for $b>2$. And in fact we know from the results of Section \ref{sec:bgrand} that the contribution of these terms is not trivial, since a drift term should appear.
\end{remark}

\subsection{Proofs of convergence and Boltzmann-Gibbs principle}

In this section we prove the above three following convergence results.

\begin{lemma}[Fluctuation part]
If $b<1$, then
$\lim_{n \to + \infty} |U_t^n (f,h)| =0.$
\end{lemma}

\begin{proof}
For each $z \in \ZZ$ we have that
\begin{equation*}
\begin{split}
n^a \int_0^t ({\mc L}_n \tau_z u) (sn^a) ds = (\tau_z u) (tn^a) - (\tau_z u) (0) +\mathcal{N}_{t}^{z,n},
\end{split}
\end{equation*}
where $\mathcal{N}^{z,n}$ is a martingale equal to $0$ at $t=0$ so that for any $t\geqslant 0$,
\[ \langle \;  \mathcal{N}_t^{z,n} \, , \, \omega^2_0 (0) -\beta^{-1}\; \rangle_{\beta} = 0.\]
Notice that the expression for $u$ only involves terms of the form $\omega_x\omega_y$ with $x\neq y$, so that $\langle \tau_z u (0) \, (\omega_0^2 (0) -\beta^{-1}) \rangle_{\beta}=0.$
We use the translation invariance of the dynamics in order to write
\begin{equation*}
\begin{split}
U_t^n (f,h)&= \tfrac{1}{n^2} \sum_{y,z \in \ZZ} (\nabla_n f) \big( \tfrac{y+z}{n} \big) h\big( \tfrac{y}{n}\big) \Big\langle \; \left[ (\tau_z u) (tn^a) - (\tau_z u) (0) \right]  \, , \, \omega^2_0 (0) -\beta^{-1}\;  \Big\rangle_{\beta}\\
&=\tfrac{1}{n}\left\langle  \left[ \tfrac{1}{\sqrt n} \sum_{y \in \ZZ} h\big( \tfrac{y}{n}\big) \Big( \omega^2_y (0) -\beta^{-1} \Big) \right] \left[ \tfrac{1}{\sqrt n} \sum_{z \in \ZZ} (\nabla_n f)\big( \tfrac{z}{n}\big) (\tau_z u) (tn^a) \right] \right\rangle_{\beta}.
\end{split}
\end{equation*}
Then, by using Cauchy-Schwarz's inequality and the stationarity of $\mu_\beta$, we have that there exist constants $C', C>0$ independent of $n$ such that
\begin{equation*}
\begin{split}
\left| U_t^n (f,h) \right|  &\le\tfrac{C'}{n} \bigg\| \tfrac{1}{\sqrt n} \sum_{y \in \ZZ} h\big( \tfrac{y}{n}\big) \Big( \omega^2_y (0) -\beta^{-1} \Big)\bigg\|_{{\bf L}^2 (\mu_\beta)}  \bigg\|  \, \tfrac{1}{\sqrt n} \sum_{z \in \ZZ} (\nabla_n f)\big( \tfrac{z}{n}\big)  \tau_z u  \,  \bigg\|_{{\bf L}^2 (\mu_\beta)} \\
& \le\tfrac{C}{n} \bigg\|  \, \tfrac{1}{\sqrt n} \sum_{z \in \ZZ} (\nabla_n f)\big( \tfrac{z}{n}\big)  \tau_z u  \,  \bigg\|_{{\bf L}^2 (\mu_\beta)}.
\end{split}
\end{equation*}
Using \eqref{eq:uu} and Parseval's relation, a simple computation shows that
\begin{equation*}
\begin{split}
\bigg\|  \, \frac{1}{\sqrt n} \sum_{z \in \ZZ} (\nabla_n f)\big( \tfrac{z}{n}\big)  \tau_z u  \,  \bigg\|_{{\bf L}^2 (\mu_\beta)}^2 &= \frac{1}{\beta^2 n}\;  \sum_{k\geqslant 1} \sum_{y, x \in \ZZ} (\nabla_n f) \big( \tfrac{y}{n}\big)(\nabla_n f) \big( \tfrac{y-x}{n}\big) \; \bigg\{ \sum_{z \in \ZZ} \rho_k(z) \rho_k (z-x)\bigg\}\\
&= \cfrac{1}{\beta^2 n}\;  \sum_{k\geqslant 1} \int_{\TT} |{\widehat \rho}_k (\theta)|^2 \bigg[ \sum_{y, x \in \ZZ} e^{-2i\pi \theta x} (\nabla_n f) \big( \tfrac{y}{n}\big)(\nabla_n f) \big( \tfrac{y-x}{n}\big) \bigg] \, d\theta\\
&=\cfrac{n}{\beta^2}\;  \sum_{k\geqslant 1} \int_{\TT} |{\widehat \rho}_k (\theta)|^2 |\cF_n(\nabla_n f)|^2 (n\theta) \, d\theta \\
&=  \cfrac{n}{\beta^2}\; \int_{\TT}\frac{|{\widehat \rho}_1 (\theta)|^2}{1-|X(\theta)|^2}  \, |\cF_n{(\nabla_n f)}|^2 (n\theta) \; d\theta.
\end{split}
\end{equation*}
In Appendix \ref{sec:fourier}, the Fourier transform of a function $g :\tfrac{1}{n} \Z \to \R$ is defined by
\begin{equation} \label{eq:FFF}
\cF_n(g) (k) = \tfrac{1}{n} \sum_{x \in \Z} g (\tfrac{x}{n}) e^{2i \pi kx/n}, \quad k\in \R.
\end{equation}
According to Lemma \ref{lem:estimates},  we know that there exists a constant $C>0$ such that
%
%
\begin{equation*}
\frac{|{\widehat \rho}_1 (\theta)|^2}{1-|X(\theta)|^2} \le  \cfrac{C} {\sqrt{\gamma_n}\Big[ \tfrac{\gamma_n}{\lambda} +\sin^2 (\pi \theta) \Big]}.
\end{equation*}
By Lemma \ref{lem:sfp}, $ |\cF_n(\nabla_n f)|^2 (n\theta)$ is bounded above by a constant $C>0$ independent of $n$ and $\theta$, from where we deduce that
\begin{equation*}
\begin{split}
\left\|  \, \frac{1}{\sqrt n} \sum_{z \in \ZZ} (\nabla_n f)\big( \tfrac{z}{n}\big)  \tau_z u  \,  \right\|_{{\bf L}^2 (\mu_\beta)}^2 & \le\cfrac{ Cn}{\sqrt{\gamma_n}} \int_{\TT}  \cfrac{1} {\left[ \tfrac{\gamma_n}{\lambda} +\sin^2 (\pi \theta) \right]}\, d\theta.
\end{split}
\end{equation*}
The RHS of the last inequality is of order $n/\gamma_n$ so that the lemma follows as soon as $\sqrt{n \gamma_n}$ diverges to $\infty$, which is equivalent to the condition $b<1$.
\end{proof}

Now we deal with the term $K_t^n(f,h)$.

\begin{lemma}[Diffusive behavior]
If $a=2- b/2$ and $b<2$, then
\[ \lim_{n\to\infty} \left\vert K_t^n(f,h)-\int_0^t \sigma_s^n(\kappa f'',h) \, ds\right\vert =0.\]
\end{lemma}

\begin{proof} First, let us write
\begin{multline*}
K_t^n(f,h)-\int_0^t \sigma_s^n(\kappa f'',h) \, ds \\=  \int_0^t \tfrac{1}{n} \sum_{x,y\in \ZZ} \Big[-n^{a-1}F\big(\tfrac{x}{n}\big)-\kappa f''\big(\tfrac{x}{n}\big)\Big]h\big(\tfrac{y}{n}\big) \left\langle \big(\omega^2_x(tn^a)-\beta^{-1} \big)\big(\omega^2_y(0)-\beta^{-1}\big)\right\rangle_\beta.
\end{multline*}
Then, from the Cauchy-Schwarz inequality, we get that there exists a constant $C>0$ such that
\[\left\vert K_t^n(f,h)-\int_0^t \sigma_s^n(\kappa f'',h) \, ds  \right\vert^2 \le   \cfrac{C}{n} \sum_{w \in \ZZ} \Big[  \big(n^{a-1} F\big)\big( \tfrac{w}{n} \big) +  \kappa \big( f'' \big)\big( \tfrac{w}{n} \big)  \Big]^2.\]
We are reduced to prove that the RHS vanishes, as $n$ goes to $\infty$. The proof relies on the Fourier transform. The discrete Fourier transform of $n^{a-1} F$ is given by
\begin{equation*}
\begin{split}
n^{a-1} \cF_n(F) (\xi) &= n^{a-2} \sum_ {y\in\mathbb{Z}} F (\tfrac{y}{n}) e^{2i \pi \xi y/n} = n^{a-2} \cF_n{(\Delta_n f)} (\xi)  {\widehat \rho_1} (\tfrac{\xi}{n})\\
&= - n^a \; \cfrac{4\sin^2 (\pi \tfrac{\xi}{n}) }{\gamma_n +\lambda \sqrt{(1+\tfrac{\gamma_n}{\lambda} )^2 -\cos^2 (\pi \tfrac{\xi}{n})}} \cF_n(f)(\xi).
\end{split}
\end{equation*}
We denote
\[
q_n \big( \tfrac{\xi}{n}\big):=- n^a \; \cfrac{4\sin^2 (\pi \tfrac{\xi}{n}) }{\gamma_n +\lambda \sqrt{(1+\tfrac{\gamma_n}{\lambda} )^2 -\cos^2 (\pi \tfrac{\xi}{n})}} .
\]
Remind that $\kappa=\tfrac{1}{\sqrt{2\lambda c}}$. By Plancherel's relation it is equivalent to prove
\begin{equation*}
\lim_{n \to \infty} I_n := \lim_{n \to \infty}  \int_{\big[-\tfrac{n}{2}, \tfrac{n}{2}\big]} \big| q_n \big( \tfrac{\xi}{n}\big) \cF_n(f) (\xi)- \tfrac{1}{\sqrt{{2\lambda c}}} \cF_n(f'') (\xi) \big|^2 d\xi =0.
\end{equation*}
Since
$$\lim_{n \to \infty} \int_{\big[-\tfrac{n}{2}, \tfrac{n}{2}\big]} | \cF_n(g) (\xi) -( {\mc F} g) (\xi)|^2\, d\xi =0$$
for any $g \in {\mc S} (\RR)$, we can replace in $I_n$ the term $\cF_n(f'') (\xi)$ by ${\mc F} (f'') (\xi) = - 4 \pi^2 |\xi|^2 ({\mc F} f)(\xi)$, where ${\mc F}$ is the usual Fourier transform defined on ${\mc S} (\RR)$. We write then
\begin{equation*}
\begin{split}
& \int_{\big[-\tfrac{n}{2}, \tfrac{n}{2}\big]} \bigg| q_n \big( \tfrac{\xi}{n}\big) \cF_n(f) (\xi)+   \tfrac{4 \pi^2 \xi^2 }{\sqrt{{2\lambda c}}} (\mc F f) (\xi) \bigg|^2 d\xi\\ & \le 2  \int_{\big[-\tfrac{n}{2}, \tfrac{n}{2}\big]} \bigg| q_n \big( \tfrac{\xi}{n}\big)+ \tfrac{4 \pi^2 \xi^2 }{\sqrt{{2\lambda c}}}\bigg|^2 \, \big| \cF_n(f) (\xi) \big|^2 d\xi + \tfrac{16 \pi^4}{\lambda c}  \int_{\big[-\tfrac{n}{2}, \tfrac{n}{2}\big]} \xi^4 |\cF_n(f) (\xi) -({\mc F} f)(\xi)|^2 d\xi.
\end{split}
\end{equation*}
The last term of the RHS of the previous inequality goes to $0$, as $n \to \infty$, since $f\in{\mc S} (\RR)$. We are reduced to show that
$$\lim_{n \to \infty} \int_{\big[-\tfrac{n}{2}, \tfrac{n}{2}\big]} \bigg| q_n \big( \tfrac{\xi}{n}\big)+  \tfrac{4 \pi^2 \xi^2}{\sqrt{{2\lambda c}}} \bigg|^2 \big| \cF_n(f)  (\xi) \big|^2 d\xi =0.$$
Since $f\in{\mc S} (\RR)$, a simple application of Lemma \ref{lem:sfp} shows that it is equivalent to prove that
$$\lim_{n \to \infty} \int_{\big[-\tfrac{n}{2}, \tfrac{n}{2}\big]} \bigg| q_n \big( \tfrac{\xi}{n}\big)+   \tfrac{4n^2}{\sqrt{{2\lambda c}}} \sin^2(\pi \tfrac{\xi}{n}) \bigg|^2 \big| \cF_n(f)  (\xi) \big|^2 d\xi =0.$$
Observe now that
\begin{multline}
\label{eq:quest}
\bigg| q_n \big( \tfrac{\xi}{n}\big)+  \tfrac{4n^2}{\sqrt{{2\lambda c}}} \sin^2(\pi \tfrac{\xi}{n}) \bigg| \\ =   \tfrac{4n^2}{\sqrt{{2\lambda c}}} \sin^2(\pi \tfrac{\xi}{n}) \left|\cfrac{1}{\sqrt{1+ \tfrac{cn^{-b}}{2\lambda} + \tfrac{n^b \lambda}{2c}\sin^2 (\pi \tfrac{\xi}{n})   } +\sqrt{\tfrac{cn^{-b}}{2\lambda}} } - 1 \right|.
\end{multline}
In particular, we have
\begin{equation}
\label{eq:quest2}
\big| q_n \big( \tfrac{\xi}{n}\big)+  \tfrac{ 4n^2 }{\sqrt{{2\lambda c}}}  \sin^2(\pi \tfrac{\xi}{n}) \big| \le C  |\xi|^2.
\end{equation}
Observe that by Lemma \ref{lem:sfp} we have
\begin{equation*}
\limsup_{A \to \infty}\;  \limsup_{n \to \infty} \int_{A \le |\xi| \le n/2} |\xi|^4 \big| \cF_n(f)  (\xi) \big|^2 d\xi =0.
\end{equation*}
Thus, by \eqref{eq:quest2}, it is sufficient to prove that for any $A>0$ fixed,
$$\lim_{n \to \infty} \int_{|\xi|\le A} \bigg| q_n \big( \tfrac{\xi}{n}\big)+  \tfrac{ 4n^2 }{\sqrt{{2\lambda c}}} \sin^2(\pi \tfrac{\xi}{n}) \bigg|^2 \big| \cF_n(f)  (\xi) \big|^2 d\xi =0.$$
By \eqref{eq:quest} we have
\begin{multline*}
\int_{|\xi|\le A} \bigg| q_n \big( \tfrac{\xi}{n}\big)+ \tfrac{ 4n^2 }{\sqrt{{2\lambda c}}}  \sin^2(\pi \tfrac{\xi}{n}) \bigg|^2 \big| \cF_n(f)  (\xi) \big|^2 d\xi \\ \le {\ve}_n (A) \int_{|\xi| \le A} \big| \cF_n(f) (\xi)|^2 d\xi \le \cfrac{\ve_n (A)}{n} \sum_{x \in \ZZ} f^2 (\tfrac{x}{n})
\end{multline*}
where
\begin{equation*}
{\ve}_n (A) = \sup_{|\xi| \le A} \left\{  \tfrac{ 4n^2 }{\sqrt{{2\lambda c}}} \sin^2(\pi \tfrac{\xi}{n}) \left|\cfrac{1}{\sqrt{1+ \tfrac{cn^{-b}}{2\lambda} + \tfrac{n^b \lambda}{2c}\sin^2 (\pi \tfrac{\xi}{n})   } +\sqrt{\tfrac{cn^{-b}}{2\lambda}} } - 1 \right| \right\}
\end{equation*}
goes to $0$ as $n$ goes to infinity since $b<2$. Since \[\sup_n \tfrac{1}{n} \sum_{x \in \ZZ} f^2 (\tfrac{x}{n}) <\infty\] the claim follows.
\end{proof}

\begin{lemma}[Boltzman-Gibbs principle]
If $b<2/3$, then we have
$\lim_{n \to \infty} \vert \Psi_t^n (f,h) \vert =0.$
\end{lemma}
\begin{proof}
Recall \eqref{theta}. Performing a simple computation we can rewrite $\psi$ as
\begin{equation*}
\psi= \sum_{x\in\mathbb{Z}}\sum_{k \geqslant 1} \big( \rho_{k-1} (x+1) + \rho_{k+1} (x) \big) \nabla \omega_x \omega_{x+k},
\end{equation*}
with the convention that $\rho_0 (x)=0$ for any $x\in \ZZ$. Let us introduce $\psi_k$ defined by
\begin{equation*}
\psi_k (x) =  \rho_{k-1} (x+1) + \rho_{k+1} (x).
\end{equation*}
By using space invariance of $\langle \cdot \rangle_{\beta}$ we get that
\begin{equation*}
\Psi_t^n (f,h)= n^{a-2} \left\langle \sum_{y\in\mathbb{Z}} h\big( \tfrac{y}{n}\big) (\omega_y^2 (0) -\tfrac{1}{\beta}) \; , \;  \int_0^t \varphi(\omega(sn^a))   \, ds \right\rangle_{\beta}
\end{equation*}
where the function $\varphi$ is given by
 \begin{equation}
\label{eq:Psifunction}
\begin{split}
\varphi (\omega)& = \sum_{x,z \in \ZZ} \sum_{k \geqslant 1} (\nabla_n f) \big( \tfrac{z}{n}\big) [ \psi_k (x-1) -\psi_k (x)]  (\omega_{x+z} \omega_{x+z+k} ).
\end{split}
\end{equation}

Thus, by the Cauchy-Schwarz's inequality, it is sufficient to prove that the ${\bf L}^2(\bb P_{\mu_{\beta}})$ norm of
\begin{equation*}
n^{a-3/2} \int_0^t \varphi(\omega (sn^a)) ds
\end{equation*}
vanishes with $n$, where $\bb P_{\mu_\beta}$ denotes the law of the Markov process $\{\omega(tn^a)\}_{t\geqslant 0}$ starting with $\mu_\beta$. We denote by $\bb E_{{\mu_\beta}}$ the corresponding expectation.
By a general inequality for the variance of additive functionals of Markov processes (see \cite{MR2952852}, Lemma 2.4), we have
\begin{align}
\bb E_{\mu_\beta} \left[\left( \int_0^t \varphi (\omega (sn^a))\, ds \right)^2 \right] &\le C t \left\langle \varphi \, , \, \left( t^{-1} - n^a {\mc S}_n \right)^{-1} \varphi  \right\rangle_{\beta} \notag\\
&= C t n^{-a}  \left\langle \varphi  \, , \, \left( [t n^a]^{-1} - {\mc S}_n \right)^{-1} \varphi  \right\rangle_{\beta} \label{eq:h-1psi}
\end{align}

We use some ideas from \cite{BG} in order to have a very sharp estimate of (\ref{eq:h-1psi}). In Appendix \ref{app:A}, we prove that $\langle \varphi  \, , \, ( [t n^a]^{-1} - {\mc S}_n )^{-1} \varphi  \rangle_{\beta}$ is bounded from above by $C n^{2b}$. As a consequence, the Boltzmann-Gibbs principle holds if $a-3+2b < 0$, and with the condition $a=2-b/2$ it gives $b<2/3$.
\end{proof}

\section{Superdiffusive domain of the energy}
\label{sec:bgrand}

 In this section we give the strategy of the proof of Theorem \ref{theo2}, which is the same as in \cite{BGJ}.  In the whole section, $a=3/2$, and $b>1$. We also assume $\beta=1$, the general case follows after an easy change of variables in the Markov process.

\subsection{Weak formulation}

\paragraph{\sc Two coupled differential equations} Let $g$ be a fixed function in $\mc S(\bb R)$. We define the process $\{\mc S_t^n; t \geqslant 0\}$ acting on functions $f \in \mc S(\bb R)$ as
\begin{equation}
\label{eq:stn}
\mc S_t^n(f):= \frac{1}{2}\sigma_t^n(f,g).
\end{equation}

for any $t \geqslant 0$, $n \in \bb N$. After arranging terms in a convenient way we have that
\begin{equation*}
\mc S_t^n(f) = \frac{1}{2}\left\langle \left\{\frac{1}{\sqrt n} \sum_{x \in \bb Z} g\Big(\frac{x}{n}\Big)\left(\omega_x^2(0)-1\right) \right\} \times
	\left\{ \frac{1}{\sqrt n} \sum_{y \in \bb Z}f\Big(\frac{y}{n}\Big)\left(\omega_{y}^2(tn^{3/2}) -1\right)  \right\}\right\rangle_1.
\end{equation*}
 For a function $h \in \mc S(\bb R)$ we define $\{ Q_t^n(h); t \geqslant 0\}$ as
\begin{equation*}
 Q_t^n(h) = \frac{1}{2}\Bigg\langle \Bigg\{\frac{1}{\sqrt n} \sum_{x \in \bb Z} g\left(\frac{x}{n}\right)\left(\omega_x^2(0)-1\right) \Bigg\} \times
	\Bigg\{ \frac{1}{ n} \sum_{\substack{y,z  \in \bb Z\\y \neq z}} h\left(\frac{y}{n}, \frac{z}{n}\right) \omega_{y}(tn^{3/2}) \omega_{z}(tn^{3/2})   \Bigg\}\Bigg\rangle_1.
\end{equation*}
Notice that $Q_t^n(h)$ depends only on the symmetric part of the function $h$. Therefore, we will always assume, without loss of generality, that $h(x,y) = h(y,x)$ for any $x,y \in \bb Z$. Notice as well that $Q_t^n(h)$ does not depend on the values of $h$ at the diagonal $\{x=y\}$.
Let us write now the differential equations for $\mc S_t^n(f)$ and $Q_t^n(h)$. We start by introducing some definitions.

\begin{definition}[Discrete approximations]  For  $f,h \in \mc S(\bb R)$,  we define the discrete approximation
\begin{enumerate}[(i)]
\item  $\Delta_n f : \bb R \to \bb R$ of the second derivative of $f$ as
\begin{equation*}
\Delta_n f \big( \tfrac{x}{n}\big) = n^2\Big( f\big(\tfrac{x+ 1}{n} \big) + f\big( \tfrac{x- 1}{n} \big) - 2 f\big(\tfrac{x}{n}\big) \Big).
\end{equation*}
\item  $\nabla_n f \otimes \delta :  \frac{1}{n} \bb Z^2 \to \bb R$ of the distribution $f\,'(x) \delta(x=y)$ as
\begin{equation}
\label{eq:3.9}
\big(\nabla_n f \otimes \delta\big) \big( \tfrac{x}{n}, \tfrac{y}{n}\big) =
\begin{cases}
\frac{n^2}{2}\big(f\big(\tfrac{x+1}{n}\big) - f\big(\tfrac{x}{n}\big)\big); & y =x+ 1\\
\frac{n^2}{2}\big(f\big(\tfrac{x}{n}\big) - f\big(\tfrac{x-1}{n}\big)\big); & y =x- 1\\
0; & \text{ otherwise.}
\end{cases}
\end{equation}

\item  $\Delta_n h : \bb R^2 \to \bb R$ of the Laplacian of $h$ as
\begin{equation*}
\Delta_n h\big( \tfrac{x}{n}, \tfrac{y}{n}\big) = n^2\Big( h\big( \tfrac{x+1}{n}, \tfrac{y}{n}\big)+h\big( \tfrac{x-1}{n}, \tfrac{y}{n}\big)+h\big( \tfrac{x}{n}, \tfrac{y+1}{n}\big)+ h\big( \tfrac{x}{n}, \tfrac{y-1}{n}\big) - 4 h\big( \tfrac{x}{n}, \tfrac{y}{n}\big)\Big),
\end{equation*}

\item  $A_n h: \bb R \to \bb R$ of the directional derivative $(-2,-2) \cdot \nabla h$ as
\begin{equation*}
A_n h\big( \tfrac{x}{n}, \tfrac{y}{n}\big) = n\Big(h\big( \tfrac{x}{n}, \tfrac{y-1}{n}\big)+ h\big( \tfrac{x-1}{n}, \tfrac{y}{n}\big)- h\big( \tfrac{x}{n}, \tfrac{y+1}{n}\big)-h\big( \tfrac{x+1}{n}, \tfrac{y}{n}\big)\Big),
\end{equation*}

\item  $\mc D_n h : \sfrac{1}{n} \bb Z \to \bb R$ of the directional derivative of $h$ along the diagonal $\{x=y\}$ as
\begin{equation*}
\mc D_n h\big( \tfrac{x}{n} \big) = n \Big( h \big(\tfrac{x}{n}, \tfrac{x+1}{n}\big) - h \big( \tfrac{x-1}{n}, \tfrac{x}{n} \big) \Big),
\end{equation*}

\item  $\tilde {\mc D}_n h :\sfrac{1}{n} \bb Z^2 \to \bb R$ of the distribution $\partial_y h(x,x) \otimes \delta(x=y)$ as
\begin{equation*}
\widetilde{\mc D}_n h (\tfrac{x}{n},\tfrac{y}{n}) =
\begin{cases}
n^2 \big(h\big(\tfrac{x}{n}, \tfrac{x+1}{n}\big)-h\big(\tfrac{x}{n}, \tfrac{x}{n}\big)\big); & y =x+1\\
n^2 \big(h\big(\tfrac{x-1}{n}, \tfrac{x}{n}\big)-h\big(\tfrac{x-1}{n}, \tfrac{x-1}{n}\big)\big); & y=x-1\\
0; & \text{ otherwise.}
\end{cases}
\end{equation*}
\end{enumerate}
\end{definition}

The following proposition can be deduced after straightforward computations:

\begin{proposition}\label{prop:equadiff} For any function $f\in\mc S (\bb R)$, any function $h \in \mc S(\bb R^2)$,
\begin{align}
\frac{d}{dt} \mc S_t^n(f) &= -2Q_t^n(\nabla_n f \otimes \delta) +  \mc S_t^n\left(n^{-1/2}\, \Delta_n f\right), \label{ec3.10}\\
\frac{d}{dt} Q_t^n(h) &= Q_t^n\big(L_n h \big) -2 \mc S_t^n\big( \mc D_n h\big) + 2 Q_t^n\big(n^{-1/2}\, \tilde{\mc D}_n h\big),\label{ec3.15}
\end{align}
where the operator $L_n$ is defined by
\begin{equation}L_n ={\sqrt n}\,  A_n + \frac{1}{\sqrt n}\,  \Delta_n - 4n^{3/2} \gamma_n {\rm Id}. \label{eq:operator}\end{equation}
\end{proposition}

\paragraph{\sc A priori bounds} For $f \in \mc S(\bb R)$, define the weighted $\ell^2(\bb Z)$-norm as
\begin{equation*}
\|f\|_{2,n} = \sqrt{\vphantom{h^h_h}\smash{\tfrac{1}{n} \sum_{x \in \bb Z} f\big( \tfrac{x}{n} \big)^2}}.
\end{equation*}
By the Cauchy-Schwarz's inequality we have the {\em a priori} bound
\begin{equation}
\label{ec3.4}
\big| \mc S_t^n(f) \big| \leq \|g\|_{2,n} \| f\|_{2,n}
\end{equation}
for any $t \geqslant 0$, any $n \in \bb N$ and any $f,g \in \mc S(\bb R)$. Therefore, the term $ \mc S_t^n(\frac{1}{\sqrt n} \Delta_n f)$ is negligible, as $n \to \infty$. In \eqref{ec3.10}, the term $Q_t^n( \nabla_n f \otimes \delta)$ is the relevant one.  We also have the {\em a priori} bound
\begin{equation}
\label{ec3.6}
\big| Q_t^n(h) \big| \leq 2 \|g\|_{2,n} \|{\bar h} \|_{2,n},
\end{equation}
where $\|{\bar h}\|_n$ is the weighted $\ell^2(\bb Z^2)$-norm of $\bar h$
\begin{equation*}
\|{\bar h}\|_{2,n} = \sqrt{\vphantom{h^h_h}\smash{\tfrac{1}{n^2} \sum_{x, y \in \ZZ} {\bar h} \big(\tfrac{x}{n} , \tfrac{y}{n}\big)^2}}
\end{equation*}
and $\bar h$ is defined by
\begin{equation*}
{\bar h} \big(\tfrac{\vphantom{y}x}{n} , \tfrac{y}{n}\big) = h \big(\tfrac{\vphantom{y}x}{n} , \tfrac{y}{n}\big) \, {\bf 1}_{x \ne y}.
\end{equation*}

In equation \eqref{ec3.10}, both fields $\mc S_t^n$ and $Q_t^n$ appear with non-negligible terms. Moreover, the term involving $Q_t^n$ is quite singular, since it involves an approximation of a distribution. Let us give the clever strategy explained in \cite{BGJ}: given $f \in \mc S(\bb R)$, if we choose $h$ such that $L_nh=\nabla_n f\otimes\delta$, we can try to cancel the term $Q_t^n( \nabla_n f \otimes \delta)$ and $Q_t^n(L_n h)$ out. Then the term $\mc S_t^n(\mc D_n h)$ provides a non-trivial drift for the differential equation \eqref{ec3.10} and the term $Q_t^n( n^{-1/2} \tilde{\mc D}_n h)$ turns out to be negligible.

\subsection{Sketch of the Proof}

After giving the topological setting needed for the Theorem \ref{theo2}, we sketch the main steps of its proof, which are detailed in \cite{BGJ}.

\paragraph{\sc Topological setting} We fix a finite time-horizon $T>0$. Let us define the {\em Hermite polynomials} $H_\ell: \bb R \to \bb R$ as
\begin{equation*}
H_\ell(x) = (-1)^\ell e^{\frac{x^2}{2}} \frac{d^\ell}{dx^\ell} \Big[e^{-\frac{x^2}{2}}\Big]
\end{equation*}
for any $\ell \in \bb N_0$ and any $x \in \bb R$. We define the {\em Hermite functions} $f_\ell: \bb R \to \bb R$ as
\begin{equation*}
f_\ell(x) = \tfrac{1}{\sqrt{\ell! \sqrt{2 \pi}}} H_\ell(x) e^{-\frac{x^2}{4}}
\end{equation*}
For any $\ell \in \bb N_0$ and any $x \in \bb R$. The Hermite functions $\{f_\ell; \ell \in \bb N_0\}$ form an orthonormal basis of ${\bf L}^2(\bb R)$. For each $k \in \bb R$, we define the {\em Sobolev space} $\mc H_k$ as the completion of $\mc C_c^\infty(\bb R)$ with respect to the norm $\|\cdot\|_{\mc H_k}$ defined as
\begin{equation*}
\|g\|_{\mc H_k} = \sqrt{\vphantom{h^h_h}\smash{\sum_{\ell \in \bb N_0} (1+\ell)^{2k} \<f_\ell,g\>^2 } }
\end{equation*}
for any $g \in \mc C_c^\infty(\bb R)$. Here we use the notation $\<f_\ell,g\> = \int g(x) f_\ell(x) dx$.
Let us denote by $\mc C([0,T]; \mc H_{k})$ the space of continuous functions from $[0,T]$ to $\mc H_k$.

\paragraph{\sc Main steps of the proof} First, we need to show tightness, and then to characterize the limit points of weakly converging subsequences.

\begin{enumerate}
\item \textsc{Tightness}.
The same standard arguments exposed in \cite{BGJ} imply the following

\begin{lemma}
\label{l2}
For any $k>\frac{19}{24}$, the sequence $\{\mc S^n_t; t \in [0,T]\}_{n \in \bb N}$ is weakly relatively compact in $\mathbf L^2([0,T]; \mc H_{-k})$. Moreover, for any $t \in [0,T]$ fixed, the sequence $\{\mc S_t^n; n \in \bb N\}$ is sequentially, weakly relatively compact in $\mc H_{-k}$.
\end{lemma}

\item \textsc{Characterization of limit points}. Fix $k >\frac{19}{24}$ and let us consider a limit point of $\{\mc S_t^n; t \in [0,T]\}_{n \in \bb N}$   with respect to the weak topology of $\mathbf L^2([0,T]; \mc H_{-k})$.  Without loss of generality we can denote by $n$ the subsequence for which $\{\mc S_t^n; t \in [0,T]\}_{n \in \bb N}$ converges to $\{\mc S_t; t \in [0,T]\}$.
The aim is to prove the following

\begin{proposition}\label{prop:unique}
Let $f: [0,T] \times \bb R \to \bb R$ be a smooth function of compact support (in $\mc C_c^\infty([0,T] \times \bb R)$). Then,
\begin{equation}
\label{ec4.43}
\mc S_T(f_T) = \mc S_0(f_0) + \int_0^T \mc S_t\big((\partial_t+\bb L)f_t\big) dt,
\end{equation}
where $\bb L$ is defined in \eqref{frac ope}.
\end{proposition}
Then, $\{\mc S_t; t \in [0,T]\}$ is a {\em weak solution} of the fractional heat equation:
\begin{equation*}
\partial_t u =  -\tfrac{1}{\sqrt{2}}\big\{ (-\Delta)^{3/4} + \nabla (-\Delta)^{1/4}\big\} u,
\end{equation*} as defined in (2.1) of \cite{J}. In Section 8.1 of \cite{J}, it is shown that there is a unique solution of \eqref{ec4.43} and therefore the limit process $\{\mc S_t; t \in[0,T]\}$ is unique. Proposition \ref{prop:unique} is the most challenging part of the proof, and the next section is devoted to it.

\item \textsc{Conclusion.} The proof of Theorem \ref{theo2} is almost done: the first two points imply that the sequence $\{\mc S_t^n; t \in [0,T]\}_{n \in \bb N}$ weakly converges in $\mathbf L^2([0,T]; \mc H_{-k})$  to a unique limit point,  denoted by $\{\mc S_t; t \in [0,T]\}$.

It can be proved that the convergence also holds for fixed times $t \in [0,T]_{\bb Q}$ with respect to the weak topology of $\mc H_{-k}$, where \[[0,T]_{\bb Q} := \left\{t \in [0,T]; \frac{t}{T} \in \bb Q\right\}.\] Since $T$ is arbitrary, this last convergence holds for any $t \in [0,\infty)$.  In particular, $\mc S_t^n(f)$ converges to $\mc S_t(f)$, as $n \to \infty$, for any $f \in \mc C_c^\infty(\bb R)$.

\end{enumerate}

%

The main differences between the model in \cite{BGJ} and ours rely on the velocity-flip noise, of intensity $\gamma_n$. This additional perturbation first appears in the definition of the operator $L_n$ in \eqref{eq:operator}. Then, some technical proofs have to be slightly modified. More precisely, rigorous convergence estimates lead to the condition: $b>1$.

\subsection{Convergence estimates}

 Here we give the proof of Proposition \ref{prop:unique}. Recall that $b>1$. Let us assume that $\{\mc S_t^n; t \in [0,T]\}_{n \in \bb N}$ converges to $\{\mc S_t; t \in [0,T]\}$ with respect to the weak topology of  $\mathbf L^2([0,T]; \mc H_{-k})$.

 Let us fix $f \in \mc S(\bb R)$ and let $h_n : \frac{1}{n} \bb Z \times  \frac{1}{n} \bb Z \to \bb R$ be the solution of the equation
\begin{equation}
\label{ec4.22}
L_n h_n= \nabla_n f \otimes \delta.
\end{equation}
The following properties of $h_n$ are shown in Appendix \ref{appc}, following \cite{BGJ}:

\begin{lemma}
\label{l3}
Let $f \in \mc S(\bb R)$.
The solution of \eqref{ec4.22} satisfies
\begin{align}
\lim_{n \to \infty} \ & \tfrac{1}{n^2} \sum_{x,y \in \bb Z} h_n^2\big(\tfrac{x}{n}, \tfrac{y}{n}\big) =0, \label{eq:l31}\\
\lim_{n \to \infty} \ & \tfrac{1}{n} \sum_{x \in \bb Z} \big|\mc D_n h_n \big(\tfrac{x}{n}\big) +\tfrac{1}{4}\bb L f\big(\tfrac{x}{n}\big) \big|^2 =0. \label{eq:l32}
\end{align}
In other words, $\|h_n\|_{2,n}$ and $\big\|\mc D_n h_n +\frac{1}{4} \bb Lf \big\|_{2,n}$ converge to $0$, as $n \to \infty$.
\end{lemma}
By \eqref{ec3.10} and \eqref{ec3.15}, we see that
\begin{multline*}
\mc S_T^n(f)= \mc S_0^n(f) + \int_0^T \mc S_t^n\big(-4\mc D_n h_n\big) dt+ 2 \big[Q_0^n(h_n) - Q_T^n(h_n)\big] \\
+ \int_0^T \mc S_t^n\big(\tfrac{1}{\sqrt n}\Delta_n f \big) dt + 4\int_0^T {Q}_t^n \big( \tfrac{1}{\sqrt n}\tilde{\mc D_n}(h_n) \big) dt.
\end{multline*}
Therefore, by the {\em a priori} bound \eqref{ec3.6} and Lemma \ref{l3}, we have that
\begin{equation}
\label{ec4.277}
\mc S_T^n(f) = \mc S_0^n(f) + \int_0^T \mc S_t^n(\bb L f) dt + 4\int_0^T {Q}_t^n \big( \tfrac{1}{\sqrt n}\tilde{\mc D_n}(h_n) \big) dt
\end{equation}
plus an error term which goes to $0$, as $n \to \infty$. It turns out that the {\em a priori} bound \eqref{ec3.6} is not sufficient to show that the last term on the righthand side of  \eqref{ec4.277} goes to $0$ with $n$ since
\begin{equation}
\label{eq:l33}
\tfrac{1}{n^3} \sum_{x \in \bb Z} (\tilde{\mc D}_n h_n)^2 \big(\tfrac{x}{n}, \tfrac{x+1}{n}\big)
\end{equation}
is of order one. Therefore we use again \eqref{ec3.15} applied to $h=v_n$ where $v_n$ is the solution of the Poisson equation
\begin{equation}
\label{eq:poisson-v}
L_n v_n =n^{-1/2} {\tilde {\mc D}}_n h_n.
\end{equation}
Then we have
\begin{equation*}
\int_0^T Q_t^n \big( \tfrac{1}{\sqrt n} {\tilde{\mc D}}_n h_n \big) dt = 2 \int_0^T {\mc S}_t^n ({\mc D}_n v_n) dt -2 \int_{0}^T {Q}_t^n  \big( \tfrac{1}{\sqrt n} {\tilde{\mc D}}_n v_n \big) dt + Q_T^n (v_n) -Q_0^n (v_n).
\end{equation*}
Now, we can use the {\em a priori} bound \eqref{ec3.4} and \eqref{ec3.6}. The following estimates on $v_n$ are proved in Appendix \ref{appc}.
\begin{lemma}
\label{lem:5,7.}
The solution $v_n$ of \eqref{eq:poisson-v} satisfies
\begin{align}
\lim_{n \to \infty} \ & \tfrac{1}{n^2}\sum_{x,y\in\bb Z}v_n^2\big(\tfrac {x}{n},\tfrac{y}{n}\big)=0, \label{est vn} \\
\lim_{n \to \infty} \ &\tfrac{1}{n}\sum_{x\in\bb Z}(\mc D_n v_n)^2\big(\tfrac {x}{n}\big)=0, \label{est der vn} \\
\lim_{n \to \infty} \ &\tfrac{1}{n^3}\sum_{x\in\bb Z}(\widetilde{\mc D}_n v_n)^2\big(\tfrac {x}{n},\tfrac {x+1}{n}\big)=0.\label{est tilde d vn}
\end{align}
In other words,  $ \| v_n \|_{2,n}\ ,$ $\| {\mc D}_n  v_n \|_{2,n}$ and $\big\| \tfrac{1}{\sqrt n} {\widetilde {\mc D}}_n v_n \big\|_{2,n}$ converge to $0$, as $n\to\infty$.
\end{lemma}
It follows that
\begin{equation}
\label{ec4.27}
\mc S_T^n(f) = \mc S_0^n(f) + \int_0^T \mc S_t^n(\bb L f) dt
\end{equation}
plus an error term which goes to $0$, as $n \to \infty$. Recall that $\{\mc S_t^n; t \in [0,T]\}_{n \in \bb N}$ weakly converges to $\{\mc S_t; t \in [0,T]\}$. The main difficulty is that the operator $\bb L$ is an integro-differential operator with heavy tails (in other words,  even for $f \in \mc C_c^\infty(\bb R)$ the function $\bb L f$ has heavy tails). As a result, we cannot take the limit $n \to \infty$ in \eqref{ec4.27}. Bernardin et al. in \cite{BGJ} achieved the Proposition \ref{prop:unique} after truncature considerations, and results about the Lipschitz property of the function $t \mapsto \mc S_t(f)$. We refer the reader to their paper, and also to \cite{DG} and \cite{J} for useful properties of the fractional Laplacian.

In Appendix \ref{appc}, Lemma \ref{lem:5,7.} and Lemma \ref{l3} are proved. Here, the computations are similar to \cite{BGJ}, but we take into account the additional term due to the presence of the velocity-flip noise, and explain the needed assumption on the parameter $b$.

\section{Volume fluctuations}\label{sec:volume}

Recall that he volume $\sum_{x\in\mathbb{Z}} \omega_x$ is conserved if and only if $\gamma=0$. In this section we give the behavior of the space-time correlation functions of the volume fluctuation field defined for  $f,h\in\mc S(\bb R)$ as in \eqref{st corre of vol}.

\subsection{Explicit computations and Fourier transform}

Contrary to the energy fluctuation field, the computations are explicit. Let us introduce the following notation: \[\mathcal{V}(x,t):= \langle \omega_x(t)\omega_0(0)\rangle_\beta,\] and notice that for all $x \in \Z$ and $t>0$ \[ \frac{d}{dt}\big[ \mathcal{V}(x,t)\big]=\mc V(x+1,t)-\mc V(x-1,t)-2\gamma_n \mc V(x,t)+\lambda \big(\mc V(x+1,t)+\mc V(x-1,t)-2\mc V(x,t)\big).\]
This infinite ODE's system can be rewritten for the Fourier transform $\widehat{\cv}$ as
\[ \frac{d}{dt}\big[\widehat{\cv}(\theta,t)\big] =\big[-2i\sin(2\pi \theta)-2\gamma_n-4\lambda\sin^2(\pi\theta)\big] \widehat{\cv}(\theta,t),\qquad \theta \in \TT,\, t>0.\]
Since $\widehat\cv(\theta,0)=\beta^{-1}$ for all $\theta \in \TT$, we conclude that
\[ \widehat{\cv}(\theta,t)=\beta^{-1} e^{ t \left[-2i\sin(2\pi \theta)-2\gamma_n-4\lambda\sin^2(\pi\theta)\right]}.\]
We assume from now on that $\beta=1$. The inverse Fourier transform gives
 \begin{align}\eta_t^n(f,h)&=\tfrac{1}{n} \sum_{y,z\in\Z}f\big(\tfrac{y+z}{n}\big)h\big(\tfrac{y}{n}\big)\int_{\big[-\frac{1}{2},\frac{1}{2}\big]} \widehat\cv(\theta,tn^a)e^{-2i\pi \theta z} \, d\theta \notag\\
& = \tfrac{1}{n} \sum_{y,z\in\Z}f\big(\tfrac{y+z}{n}\big)h\big(\tfrac{y}{n}\big)\int_{\big[-\frac{1}{2},\frac{1}{2}\big]}  e^{ tn^a \left[-2i\sin(2\pi \theta)-2\gamma_n-4\lambda\sin^2(\pi\theta)\right]-2i\pi \theta z}\, d\theta \label{eq:eta} \\
& =  \tfrac{1}{n} \sum_{y\in\Z} \left\{h\big(\tfrac{y}{n}\big) \times \tfrac{1}{n} \sum_{z \in \Z}f\big(\tfrac{y+z}{n}\big)\int_{\big[-\frac{n}{2},\frac{n}{2}\big]} e^{ tn^a \left[-2i\sin\big(2\pi \frac{\xi}n\big)-2\gamma_n-4\lambda\sin^2\big(\pi\frac{\xi}{n}\big)\right]-2i\pi \frac{\xi}{n} z}\, d\xi \right\}.\notag
\end{align}
\paragraph{\bf Case $a>2$ or $b<a$ } We are now going to study the convergence of the quantity  $\cj_n(t)e^{-2tn^a\gamma_n}$ where
\[\cj_n(t):= \tfrac{1}{n} \sum_{z \in \Z}f\big(\tfrac{y+z}{n}\big)\int_{\big[-\frac{n}{2},\frac{n}{2}\big]} e^{tn^a \left[-2i\sin(2\pi \frac{\xi}n)-4\lambda\sin^2(\pi\frac\xi n)\right]-2i\pi \frac{\xi}n z}\, d\xi. \]
The modulus of the integrand in the above integral equals $ e^{-4 \lambda tn^a\sin^2(\frac{\pi\xi}n)},$ and it can be easily proved that  \[\int_{\big[-\frac{n}{2},\frac{n}{2}\big]}  e^{-4 \lambda tn^a\sin^2(\frac{\pi\xi}n)} d\xi \quad \begin{cases}\xrightarrow[n\to\infty]{} 0 & \text{ if } a>2,\\
\xrightarrow[n\to\infty]{} \displaystyle \int_\R e^{-4\lambda t \pi^2 \xi^2} d\xi & \text{ if } a=2, \\
\text{ is bounded by } n & \text{ if } a<2.\end{cases} \]
The first two statements are consequences of the dominated convergence theorem. Therefore, from  the fact that $\sup_n\big\{\frac{1}{n}\sum_{z\in\Z} \vert f\big(\frac{z}n\big) \vert\big\} <+\infty$,  we conclude: \begin{enumerate}[(i)]
\item If $a>2$, then $\eta_t^n(f,h)$ vanishes as $n\to \infty$ for any $b$ (because of the above integral),
\item If $a\leqslant 2$ and $b<a$, then $\eta_t^n(f,h)$ vanishes as $n \to \infty$ (because of  $e^{-2tn^a\gamma_n}=e^{-2t c\, n^{a-b}}$).
\end{enumerate}
This proves the statements $A(iii)$, $B(v)$ and $C(v)$ of Theorem  \ref{theo:volume}. In the following we distinguish the cases $a\leqslant 1$ and $a>1$.

\paragraph{\bf Case  $a\leqslant 1$ and $b\geqslant a$ }  Let us introduce the following notation:
\begin{equation}\label{qn}
q_n(\theta):=n^a\big[-2i\sin(2\pi\theta)-2\gamma_n-4\lambda\sin^2(\pi\theta)\big].
\end{equation}
From \eqref{eq:eta}, we get
\begin{equation} \eta_t^n(f,h)=\int_{\big[-\frac{n}{2},\frac{n}{2}\big]} \widehat\cv\big(\tfrac{\xi}n,tn^a\big) \cF_n(f)(\xi) \cF_n(h)(\xi)\, d\xi \label{eq:etafh}\end{equation}
and
\begin{align}\frac{d}{dt}\big[\eta_t^n(f,h)\big]&=\tfrac{1}{n} \sum_{y,z\in\Z}f\big(\tfrac{y+z}{n}\big)h\big(\tfrac{y}{n}\big)\int_{\big[-\frac{1}{2},\frac{1}{2}\big]} q_n(\theta)\widehat\cv(\theta,tn^a)e^{-2i\pi \theta z} \, d\theta \notag\\
& = \tfrac{1}{n} \sum_{y\in\Z} \left\{h\big(\tfrac{y}{n}\big) \times \tfrac{1}{n} \sum_{z \in \Z}f\big(\tfrac{y+z}{n}\big)\int_{\big[-\frac{n}{2},\frac{n}{2}\big]}q_n\big(\tfrac{\xi}{n}\big)\widehat\cv\big(\tfrac{\xi}{n},tn^a\big)e^{-2i\pi \frac{\xi}{n}z} \, d\xi\right\} \notag\\
& =  \int_{\big[-\frac{n}{2},\frac{n}{2}\big]} q_n\big(\tfrac{\xi}n\big)\widehat\cv\big(\tfrac{\xi}{n},tn^a\big) \cF_n(f)(\xi) \cF_n(h)(\xi)\, d\xi, \label{eq:derivee_eta}
\end{align}
%
%
Let us write down the lemma that we are going to prove.

\begin{lemma}
 \emph{\bf Proof of  $A(ii)$. } If $b\leqslant 1$ and $b=a$, then
\begin{equation}
\int_{\big[-\frac{n}{2},\frac{n}{2}\big]} \left\vert q_n\big(\tfrac{\xi}n\big) \cF_n(f)(\xi)-  \cF_n(\mathbf{1}_{b=1}\times 2f'-2cf)(\xi)\right\vert\,    \left\vert \widehat\cv\big(\tfrac{\xi}{n},tn^a\big)\cF_n(h)(\xi)\right\vert\, d\xi \xrightarrow[n\to\infty]{}0.
\end{equation}

\emph{\bf Proof of $A(i)$ $B(i)$ and $C(i)$. } If $a\in (0,1)$ and $a<b$,
\begin{equation}
\left\vert\int_{\big[-\frac{n}{2},\frac{n}{2}\big]} q_n\big(\tfrac{\xi}n\big)\widehat\cv\big(\tfrac{\xi}{n},tn^a\big) \cF_n(f)(\xi) \cF_n(h)(\xi)\, d\xi\right\vert  \xrightarrow[n\to\infty]{}0.
\end{equation}

\emph{\bf Proof of $B(iii)$ and $C(iii)$. } If $a=1$ and $b > 1$, then
 \begin{equation}
\int_{\big[-\frac{n}{2},\frac{n}{2}\big]} \left\vert q_n\big(\tfrac{\xi}n\big) \cF_n(f)(\xi)-  \cF_n(2f')(\xi)\right\vert^2\,    \left\vert\widehat\cv\big(\tfrac{\xi}{n},tn^a\big) \cF_n(h)(\xi) \right\vert^2\, d\xi \xrightarrow[n\to\infty]{}0.
\end{equation}

\end{lemma}

\begin{proof}
We are going to use some powerful results regarding the Fourier transform, and the argument will be repeated many times.

First, we need the following easy-to-prove fact: for any $a\leqslant 2$ and any  $h\in\mc S(\bb R)$, there exists a constant $C>0$ such that, for any $n$ and any $\xi \in \big[-\frac{n}{2},\frac{n}{2}\big]$,
\begin{equation}
\left\vert\cF_n(h)(\xi)\widehat\cv\big(\tfrac{\xi}{n},tn^a\big)\right\vert \leqslant C. \label{eq:bound}
\end{equation}
Let us prove for instance the first point $C(iii)$ with $b=a=1$. The other statements can be obtained with the same procedure. From \eqref{eq:etafh} and \eqref{eq:derivee_eta}, we can write
\begin{align*}
\big\vert\partial_t\eta_t^n(f,h)&-\eta_t^n(2f'-2cf,h)\big\vert\\
 &\leqslant \int_{\big[-\frac{n}{2},\frac{n}{2}\big]} \left\vert q_n\big(\tfrac{\xi}n\big) \cF_n(f)(\xi)-   \cF_n( 2f'-2cf)(\xi)\right\vert\,    \left\vert\cF_n(h)(\xi) \widehat\cv\big(\tfrac{\xi}{n},tn^a\big)\right\vert\, d\xi\\
& \leqslant C \int_{\big[-\frac{n}{2},\frac{n}{2}\big]} \left\vert q_n\big(\tfrac{\xi}n\big) \cF_n(f)(\xi)-   \cF_n(2f'-2cf)(\xi)\right\vert\,  d\xi.
\end{align*}
Then, the strategy is to replace $\cF_n(2f'-2cf)(\xi)$ with
\[\cF(2f'-2cf)(\xi)=(-4i\pi\xi-2c)(\cF f)(\xi),\]
 and this can be done thanks to Corollary \ref{cor:fourier}. We write then
\begin{multline*}
\int_{\big[-\frac{n}{2},\frac{n}{2}\big]} \left\vert q_n\big(\tfrac{\xi}n\big) \cF_n(f)(\xi)-  ( -4i\pi\xi-2c)\cF(f)(\xi)\right\vert\,  d\xi \\
\leqslant \int_{\big[-\frac{n}{2},\frac{n}{2}\big]} \left\vert q_n\big(\tfrac{\xi}n\big)+4i\pi\xi+2c\right\vert \big\vert \cF_n(f)(\xi) \big\vert \, d\xi +\int_{\big[-\frac{n}{2},\frac{n}{2}\big]} (4\pi\xi+2c) \big\vert \cF_n(f)(\xi)-(\cF f)(\xi)\big\vert \, d\xi
\end{multline*}
Again, from Corollary \ref{cor:fourier}, the last term of the RHS goes to 0, as $n \to \infty$. We are reduced to show that
\[ \lim_{n\to\infty} \int_{\big[-\frac{n}{2},\frac{n}{2}\big]} \left\vert q_n\big(\tfrac{\xi}n\big)+4i\pi\xi+2c\right\vert \big\vert \cF_n(f)(\xi) \big\vert \, d\xi  = 0.\]
It is now easy to see that if $a=b=1$, then there exists a constant $C>0$ such that
\[
\left\vert q_n\big(\tfrac{\xi}n\big)+4i\pi\xi+2c\right\vert  = \left\vert 2i \big[ 2\pi\xi - n\sin\big(2\pi\tfrac{\xi}{n}\big)\big] - 4\lambda n \sin^2\big(\pi\tfrac{\xi}{n}\big)\right\vert\leqslant C\xi^2.\]
Indeed, this is a consequence of the following fact:
$ \forall \, x \in [-\pi,\pi], $ $ \sin(x)-x \vert \leqslant \frac{\vert x \vert^3}{3}.$
We recover the same type of inequality as \eqref{eq:quest2}, and the conclusion comes straightforwardly.
\end{proof}

\paragraph{\bf Case $1<a\leqslant 2$ and $b\geqslant a$ } All the previous computations can be adapted to the translated field $\widetilde\eta_t^n$, whose definition is recalled below:
\[\widetilde\eta_t^n(f,h)=\tfrac{1}{n}\sum_{x,y\in \ZZ} f\big(\tfrac{x-2tn^a}{n}\big)h\big(\tfrac{y}{n}\big) \left\langle \omega_x(tn^a)\omega_y(0)\right\rangle_\beta.\]
The translation gives rise to a multiplicative term in the computation of the Fourier transform. In the same way, we define
$\widetilde q_n(\theta):=q_n(\theta)+n^a4i\pi\theta,$ where $q_n(\theta)$ was given in \eqref{qn}
and we can easily write
\begin{align*}
\widetilde\eta_t^n(f,h)&=\int_{\big[-\frac{n}{2},\frac{n}{2}\big]} \widehat\cv\big(\tfrac{\xi}n,tn^a\big) e^{4i\pi tn^a \frac{\xi}{n}}\cF_n(f)(\xi) \cF_n(h)(\xi)\, d\xi,\\
\frac{d}{dt}\big[\widetilde\eta_t^n(f,h)\big]&=\int_{\big[-\frac{n}{2},\frac{n}{2}\big]} \widetilde q_n\big(\tfrac{\xi}n\big)\widehat\cv\big(\tfrac{\xi}{n},tn^a\big) e^{4i\pi tn^a \frac{\xi}{n}} \cF_n(f)(\xi) \cF_n(h)(\xi)\, d\xi.
\end{align*}

\begin{lemma}

\emph{\bf Proof of $B(ii)$ and $C(ii)$. } If $a\in (1,2)$ and $a<b$,
\begin{equation}
\left\vert\int_{\big[-\frac{n}{2},\frac{n}{2}\big]} \widetilde q_n\big(\tfrac{\xi}n\big)\widehat\cv\big(\tfrac{\xi}{n},tn^a\big) e^{4i\pi tn^a \frac{\xi}{n}} \cF_n(f)(\xi) \cF_n(h)(\xi)\, d\xi\right\vert  \xrightarrow[n\to\infty]{}0.
\end{equation}

\emph{\bf Proof of $B(iv)$. } If $a\in (1,2)$ and $a=b$,
\begin{equation}
\int_{\big[-\frac{n}{2},\frac{n}{2}\big]} \left\vert \widetilde q_n\big(\tfrac{\xi}n\big) \cF_n(f)(\xi)-  \cF_n(-2cf)(\xi)\right\vert\,   \left\vert e^{4i\pi tn^a \frac{\xi}{n}} \widehat\cv\big(\tfrac{\xi}{n},tn^a\big) \cF_n(h)(\xi) \right\vert\, d\xi \xrightarrow[n\to\infty]{}0.
\end{equation}

\emph{\bf Proof of $C(iii)$. } If $a=2$ and $b\geqslant 2$, then
\begin{multline*}
\int_{\big[-\frac{n}{2},\frac{n}{2}\big]} \left\vert \widetilde q_n\big(\tfrac{\xi}n\big) \cF_n(f)(\xi)-  \cF_n(\lambda f''-\mathbf{1}_{b=2}\times 2cf)(\xi)\right\vert^2\,    \left\vert e^{4i\pi tn^a \frac{\xi}{n}} \widehat\cv\big(\tfrac{\xi}{n},tn^a\big)\cF_n(h)(\xi)\right\vert^2\, d\xi \\
\xrightarrow[n\to\infty]{}0.
\end{multline*}

\end{lemma}

\begin{proof}
We repeat he same arguments as before: first, for any $a\leqslant 2$ and any  $h\in\mc S(\bb R)$, there exists a constant $C>0$ such that, for any $n$ and any $\xi \in \big[-\frac{n}{2},\frac{n}{2}\big]$,
\begin{equation}
\left\vert\cF_n(h)(\xi)\widehat\cv\big(\tfrac{\xi}{n},tn^a\big)e^{4i\pi tn^a \frac{\xi}{n}}\right\vert \leqslant C. \label{eq:bound}
\end{equation}
Let us prove for instance the last point $C(iii)$ with $b>2$ and $a=2$. The other statements can be obtained with the same procedure. From \eqref{eq:etafh} and \eqref{eq:derivee_eta}, we can write
\begin{align*}
\big\vert\partial_t\widetilde\eta_t^n(f,h)& -\lambda\widetilde\eta_t^n(f'',h)\big\vert \\
& \leqslant \int_{\big[-\frac{n}{2},\frac{n}{2}\big]} \left\vert \widetilde q_n\big(\tfrac{\xi}n\big) \cF_n(f)(\xi)-  \cF_n(\lambda f'')(\xi)\right\vert\,    \left\vert e^{4i\pi tn^a \frac{\xi}{n}}\cF_n(h)(\xi) \widehat\cv\big(\tfrac{\xi}{n},tn^a\big)\right\vert\, d\xi\\
& \leqslant C \int_{\big[-\frac{n}{2},\frac{n}{2}\big]} \left\vert \widetilde q_n\big(\tfrac{\xi}n\big) \cF_n(f)(\xi)-  \cF_n(\lambda f'')(\xi)\right\vert\,  d\xi.
\end{align*}
As before, we are going to replace $\cF_n(f'')(\xi)$ with $\cF(f'')(\xi)=-4\pi^2\vert\xi\vert^2(\cF f)(\xi)$,  thanks to Corollary \ref{cor:fourier}. We write then
\begin{multline*}
\int_{\big[-\frac{n}{2},\frac{n}{2}\big]} \left\vert \widetilde q_n\big(\tfrac{\xi}n\big) \cF_n(f)(\xi)+  4\pi^2\lambda\xi^2\cF(f)(\xi)\right\vert\,  d\xi \\
\leqslant \int_{\big[-\frac{n}{2},\frac{n}{2}\big]} \left\vert \widetilde q_n\big(\tfrac{\xi}n\big)+4\pi^2\lambda\xi^2\right\vert \big\vert \cF_n(f)(\xi) \big\vert \, d\xi + 4\pi^2\lambda\int_{\big[-\frac{n}{2},\frac{n}{2}\big]} \xi^2 \big\vert \cF_n(f)(\xi)-(\cF f)(\xi)\big\vert \, d\xi.
\end{multline*}
Again, from Corollary \ref{cor:fourier}, the last term of the RHS goes to 0, as $n \to \infty$. We are reduced to show that
\[ \lim_{n\to\infty} \int_{\big[-\frac{n}{2},\frac{n}{2}\big]} \left\vert \widetilde q_n\big(\tfrac{\xi}n\big)+4\pi^2\lambda\xi^2\right\vert \big\vert \cF_n(f)(\xi) \big\vert \, d\xi  = 0.\]
Now, if $a=2$ and $b>2$, then there exists a constant $C>0$ such that
$\left\vert \widetilde q_n\big(\tfrac{\xi}n\big)+4\pi^2\lambda\xi^2\right\vert\leqslant C\xi^2,$
and the conclusion follows.
\end{proof}

\section*{Acknowledgements}

PG thanks CNPq (Brazil) for support through the research project ``Additive functionals of particle systems" 480431/2013-2. PG thanks  CMAT for support by ``FEDER" through the
``Programa Operacional Factores de Competitividade  COMPETE" and by
FCT through the project PEst-C/MAT/UI0013/2011. 
This work has been supported by the Brazilian-French Network in Mathematics, the French Ministry of Education through the grant ANR (EDNHS) and the fellowship L'Or\'eal-France UNESCO ``Pour les Femmes et la Science''.

\appendix

\section{Fourier analysis}
\label{sec:fourier}
In this article Fourier analysis is one of the most important tools. Actually, the Fourier transform is very useful, since it is reversible, being able to transform from one domain to the other. In the case of a periodic function, the Fourier transform can be simplified to the calculation of series coefficients. Also, when the domain is a lattice, it is still possible to recreate a version of the original Fourier transform according to the Poisson summation formula, also known as discrete-time Fourier transform.

\begin{enumerate}
\item \textbf{Fourier transform of integrable functions -- } If $f: \R \to \R$ is an integrable function, we define its Fourier transform $\mathcal{F} f:\R\to\C$ as
\begin{equation}
\label{fou tranf cont}
 \mathcal{F}f(\xi):=\int_\R f(x) e^{2i\pi\xi x} \, d x, \qquad \xi \in \R.\end{equation}

\item \textbf{Fourier transform of square summable sequences -- } If $h:\Z\to\R$ is square summable, we define its Fourier transform $\widehat{h}:\TT\to\C$ in $\mathbf{L}^2(\TT)$ as
\begin{equation}
\label{fou tranf}
 \widehat{h}(\theta):=\sum_{x\in\Z} h(x) e^{2i\pi \theta x}, \qquad \theta \in \TT.\end{equation}

\item\label{FF3} \textbf{Discrete Fourier transform of integrable functions -- } If $g: \R \to \R$ is an integrable function, we define its Fourier transform $\cF_n(g):\R\to\C$ as
\begin{equation*}
\cF_n(g) (\xi) = \frac{1}{n} \sum_{x \in \Z} g \Big(\frac{x}{n}\Big) e^{2i \pi x \tfrac{\xi}{n}}, \qquad \xi\in \R.
\end{equation*}
\end{enumerate}

These definitions can easily be extended for $d$-dimensional spaces, $d\geqslant 1$. For each Fourier transform, we have the \emph{Parseval-Plancherel's identity} between suitable norms of the involved spaces, and we also can recover the initial functions for the knowledge of their Fourier transforms by the \emph{inverse Fourier transform}.

For example, for the Fourier transform \eqref{FF3},  the Parseval-Plancherel's identity reads as
\begin{equation*}
\| g \|_{2,n}^2 := \frac{1}{n} \sum_{x \in \Z} \Big| g \Big(\frac{x}{n}\Big)\Big|^2 = \int_{\big[-\tfrac{n}{2}, \tfrac{n}{2}\big]}\;  \left| \cF_n(g) (\xi) \right|^2 \, d\xi \; =: \| {\cF_n(g)} \|^2.
\end{equation*}
The function $g$ can be recovered from the knowledge of its Fourier transform by the inverse Fourier transform of $\cF_n(g)$:
\begin{equation*}
g \Big(\frac{x}{n}\Big) = \int_{\big[-\tfrac{n}{2}, \tfrac{n}{2}\big]}\;\cF_n(g) (\xi)\;  e^{-2i \pi x \tfrac{\xi}{n}} \,  d\xi .
\end{equation*}
Now we give the properties that we need.

\begin{lemma}
\label{lem:sfp}
If $g \in {\mc S} (\R)$, then for any $p\geqslant 1$, there exists a constant $C:=C(p,f)$ such that for any $|y| \le 1/2$,
\begin{equation*}
| \cF_n(g) (ny) | \le \frac{C}{1+ (n|y|)^{p}}.
\end{equation*}
\end{lemma}

\begin{proof}
This lemma is entirely proved in \cite{BGJ}.
\end{proof}
The following result is an easy corollary of the previous lemma.
\begin{corollary}\label{cor:fourier}
If $g \in {\mc S}(\R)$ is in the Schwartz space, then, for any $p\geqslant 0$,
\[ \lim_{n\to\infty} \int_{\big[-\tfrac{n}{2},\tfrac{n}{2}\big]} |\xi|^p\, \left\vert \cF_n(g)(\xi)-(\mc F g)(\xi)\right\vert^2 {d}\xi =0,\]
and there exists a constant $C>0$ such that
\[ \int_{\big[-\tfrac{n}{2},\tfrac{n}{2}\big]} \vert \xi\vert^p \,\vert \cF_n(g)(\xi) \vert^2 {d}\xi \leqslant C.\]
\end{corollary}

\section{Sharp estimate of the resolvent norm}
\label{app:A}
\subsection{Hermite polynomials}
For simplicity, we assume $\beta=1$ and we drop the index $\beta$ from the notations. Let us decompose the generator ${\mc S}_n$ is the basis of multivariate Hermite polynomials. We denote by $(h_n)_{n \geqslant 0}$ the sequence of standard Hermite polynomials, which are orthogonal with respect to the standard Gaussian probability measure on $\RR$.

Let $\Sigma$ be the set composed of configurations  $\sigma=(\sigma_x)_{x\in \ZZ} \in \NN^{\ZZ}$ such that $\sigma_{x} \ne 0$ only for a finite number of $x$. The number $\sum_{x \in \ZZ} \sigma_x$ is called the size of $\sigma$ and  is denoted by $|\sigma|$. Let $\Sigma_n = \{ \sigma \in \Sigma \; ; \; |\sigma|=n\}$. On the set of $n$-tuples $\bx:=(x_1, \ldots,x_n)$ of $\ZZ^n$, we introduce the equivalence relation $\bx \sim \by$ if there exists a permutation $p$ on $\{1, \ldots,n\}$ such that $x_{p(i)} =y_i$ for all $i \in \{1, \ldots,n\}$. The class of $\bx$ for the relation $\sim$ is denoted by $[\bx]$ and its cardinal by $c({\bf x})$.  Then the set of configurations of $\Sigma_n$ can be identified with the set of $n$-tuples classes for $\sim$ by the one-to-one application:
\begin{equation*}
[{\bf x}]=[(x_1,\ldots,x_n)] \in \ZZ^n/ \sim \; \rightarrow \sigma^{[{\bf x}]} \in \Sigma_n
\end{equation*}
where for any $y \in \ZZ$, $(\sigma^{[\bf x]})_y= \sum_{i=1}^n {\bf 1}_{y=x_i}$. We will identify $\sigma \in \Sigma_n$ with the occupation numbers of  a configuration with $n$ particles, and $[\bf x]$ will correspond to  the positions of those $n$ particles.

To any $\sigma \in \Sigma$, we associate the polynomial function $h_{\sigma}$ given by
\begin{equation*}
h_{\sigma} (\omega) = \prod_{x \in \ZZ} h_{\sigma_x} (\omega_x).
\end{equation*}
Then, the family $\left\{ h_{\sigma} \; ; \; \sigma \in \Sigma \right\}$ forms an orthogonal basis of ${\bb L}^2 (\mu_1)$ such that
\begin{equation}
\label{eq:prod hsigma}
\int h_\sigma (\omega)  \, h_{\sigma'} (\omega) \, d\mu_{1}(\omega)  = \delta_{\sigma=\sigma'}\end{equation}
where $\delta$ denotes the Kronecker function, so that $\delta_{\sigma=\sigma'}=1$ if $\sigma=\sigma'$, otherwise it is equal to zero.

A function $F:\Sigma \to \RR$ such that $F(\sigma)=0$ if $\sigma \notin \Sigma_n$ is called a degree $n$ function. Thus, such a function is sometimes considered as a function defined only on $\Sigma_n$. A local function $F \in {\mathbf L}^2 (\mu_1)$ whose decomposition on the orthogonal basis $\{ h_{\sigma} \, ; \, \sigma \in \Sigma \}$ is given by $F=\sum_{\sigma} F(\sigma)  h_{\sigma}$ is called of degree $n$ if and only if $F$ is of degree $n$. A function $F: \Sigma_n \to \RR$ is nothing but a symmetric function $F:\ZZ^n \to \RR$ through the identification of $\sigma$ with $[\bx]$. We denote {\footnote{with some abuse of notations.}} by $\langle \cdot, \cdot \rangle$ the scalar product on $\oplus {\mathbf L}^2 (\Sigma_n)$, each $\Sigma_n$ being equipped with the counting measure. Hence, if $F,G:\Sigma \to \RR$, we have
\begin{equation*}
\langle F, G \rangle = \sum_{n\geqslant 0} \sum_{\sigma \in \Sigma_n} F_n (\sigma) G_n (\sigma) = \sum_{n \geqslant 0} \sum_{\bx \in \ZZ^n} \frac{1}{c({\bf x})} \,  F_n (\bx) G_n (\bx),
\end{equation*}
with $F_n, G_n$ the restrictions of $F,G$ to $\Sigma_n$. We recall that $c({\bf x})$ is the cardinal of $[{\bf x}]$.

If a local function $F \in {\bf L}^{2} (\mu_1)$ is written in the form $F =\sum_{\sigma \in \Sigma} F(\sigma) h_{\sigma}$ then we have
\begin{equation*}
({\mc S}_n F) (\omega) = \sum_{\sigma \in \Sigma} ({\mf S}_n F)(\sigma) h_{\sigma} (\omega)
\end{equation*}
with
\begin{equation*}
({\mf S} F)(\sigma) =\lambda \sum_{x \in \ZZ} ( F(\sigma^{x,x+1}) - F(\sigma)) - \gamma_n \sum_{x\in \ZZ} W_x (\sigma) F(\sigma),
\end{equation*}
where $\sigma^{x,x+1}$ is obtained from $\sigma$ by exchanging the occupation numbers $\sigma_x$ and $\sigma_{x+1}$ and $W_x (\sigma)$ is defined by
$W_x (\sigma)= 1-(-1)^{\sigma_x} \geqslant 0.
$
The formula for $W_x (\sigma)$ is a direct consequence of the fact that $h_n$ is even (resp. odd) if $n$ is even (resp. odd).

\subsection{$H_{1,z}$ and $H_{-1,z}$ norms}

For any $z>0$ and any $f \in {\bf L}^2 (\mu_1)$ we define the $H_{\pm 1,z}$-norm of $f$ by
\begin{equation*}
\| f\|_{\pm1,z} = \left[ \langle f \, , \, (z-{\mc S}_n)^{\pm 1} f \rangle \right]^{1/2}.
\end{equation*}
We have
\begin{equation*}
\| f \|_{1,z}^2 = z \langle f , f ,\rangle + {\mc D} (f)
\end{equation*}
where ${\mc D} (f)$ is the Dirichlet form of $f$ defined by
${\mc D} (f) = \langle f, -{\mc S}_n f \rangle.
$

If $f$ has the decomposition $f= \sum_{\sigma \in \Sigma} f(\sigma) h_{\sigma}$ then
\begin{equation}
\label{eq:df01}
{\mc D} (f) =\cfrac{1}{2} \sum_{x \in \ZZ} \sum_{\sigma \in \Sigma} \left( f(\sigma^{x,x+1}) -f(\sigma) \right)^2 + \gamma_n \sum_{\sigma \in \Sigma, x \in \ZZ} W_x (\sigma) f^2 (\sigma) .
\end{equation}
Let $\Delta_{+}= \left\{ (x,y) \in \ZZ^2 \, ; \, y \geqslant x+1\right\}$, $\Delta_- = \left\{ (x,y) \in \ZZ^2 \, ; \, y \le x-1\right\}$ and $\Delta_0 = \left\{ (x,x) \, ; \, x \in \ZZ \right\}$.

We denote by ${\bb D}$ the Dirichlet form of a symmetric simple random walk on $\ZZ^2$ where jumps from $\Delta_{\pm}$ to $\Delta_0$ and from $\Delta_0$ to $\Delta_{\pm}$ have been suppressed and jumps from $(x,x) \in \Delta_0$ to $(x \pm 1, x \pm 1) \in \Delta_0$ have been added, i.e.
 \begin{equation*}
{\bb D}(f)=\cfrac{1}{2} \sum_{|\be|=1} \,\,\sum_{\bx \in \Delta_{\pm}, \bx + \be \in \Delta_{\pm}}\!\! \!\!\!\!\left( f(\bx + \be) -f(\bx)\right)^2 +  \cfrac{1}{2} \sum_{\bx \in \Delta_0} \!\left( f(\bx \pm (1,1)) -f(\bx)\right)^2,
\end{equation*}
where $f:\ZZ^2 \to \RR$ is a symmetric function such that $\sum_{\bx \in \ZZ^2} f^2 (\bx) < \infty$.

\begin{lemma}
\label{lem:df12}
Let $f= \sum_{\sigma \in \Sigma_2} f (\sigma) h_{\sigma}$ be a local  function of degree $2$. There exists a positive constant $C$, independent of $f$ and $n$, such that
\begin{equation*}
C^{-1} \left[ {\bb D} (f) + \gamma_n \sum_{\bx \notin \Delta_0} f^2 (\bx)  \right]  \le {\mc D} (f) \le C  \left[ {\bb D} (f) + \gamma_n \sum_{\bx \notin \Delta_0} f^2 (\bx)  \right] .
\end{equation*}
\end{lemma}

\begin{proof}
Straightforward.
\end{proof}

We denote by $\Sigma_2^0$ the set of configurations $\sigma$ of $\Sigma_2$ such that $\sigma= 2 \delta_x$, $x \in \ZZ$, and $\Sigma_2^{\pm}$ the complementary set of $\Sigma_2^0$ in $\Sigma_2$, i.e. the set of configurations $\sigma \in \Sigma_2$ such that $\sigma=\delta_x + \delta_y$, $y \ne x \in \ZZ$.

Observe that the function $\varphi$ defined by \eqref{eq:Psifunction} is a function of degree $2$ with a decomposition in the form $\varphi=\sum_{\sigma \in \Sigma_2} \Phi(\sigma) h_{\sigma}$ which satisfies $\Phi (\sigma)=0$ if $\sigma \in \Sigma_2^0$. We have that
\begin{equation*}
\Big\langle  \varphi\, , \left( z - {\mc S}_n \right)^{-1}\varphi \Big \rangle = \sup_{g} \Big\{ 2 \langle \varphi, g\rangle - z \langle g\;,g\rangle - {\mc D} (g) \Big\}
\end{equation*}
where the supremum is taken over local functions $g \in {\mathbf L}^2 (\mu_1)$. Decompose $g$ appearing in this variational formula as $g=\sum_{\sigma} G(\sigma) h_{\sigma}$. Recall that $\{h_{\sigma} \, ; \, \sigma \in \Sigma\}$ are orthogonal, that the function $\varphi$ is a degree $2$ function such that $\Phi(\sigma)=0$ for any $\sigma \notin \Sigma_2^{\pm}$  and formula \eqref{eq:df01} for the Dirichlet form ${\mc D} (g)$. Thus, we can restrict this supremum over degree $2$ functions $g$ such that $G(\sigma)=0$ if $\sigma \in \Sigma_2^0$. Then, by Lemma \ref{lem:df12}, we have
\begin{equation*}
\begin{split}
\Big\langle  \varphi, \left( z -{\mc S}_n \right)^{-1} \varphi  \Big \rangle \le  C\;  \sup_{G}&  \Bigg\{  \sum_{x \ne y} \Phi(x,y) G(x,y)  - (z+\gamma_n)  \sum_{\substack{(x,y) \in \ZZ^2\\ x\neq y}} G^2 (x,y) \\
&\quad \quad \quad-C' \sum_{|\be| =1}\sum_{\substack{(x,y) \in \Delta^{\pm}\\ (x,y) + \be \in \Delta^{\pm}}}  \Big( G((x,y)+\be) -G(x,y) \Big)^2\Bigg\}
\end{split}
\end{equation*}
where $C,C'$ are positive constants, $\Delta^{\pm} = \{ (x,y) \in \ZZ^2 \, ; \, x\ne y\}$ and as usual we identify the functions defined on $\Sigma_n$ with symmetric functions defined on $\ZZ^n$.

In order to get rid of the geometric constraints appearing in the last term of the variational formula, for any symmetric function $G$ defined on the set $\Delta_{\pm}$, we denote by ${\tilde G}$ its extension to $\ZZ^2$ defined by
$${\tilde G} (x,y) =G (x,y) \; \text{if}\;  x\ne y, \quad {\tilde G} (x,x) = \cfrac{1}{4} \sum_{|\be|=1} G((x,x) +\be).$$
It is trivial that there exists a constant $C>0$ such that
\begin{equation*}
\sum_{(x,y) \in \ZZ^2} {\tilde G}^2 (x,y)\le C\sum_{\substack{(x,y) \in \ZZ^2\\ x\neq y}} {G}^{2} (x,y),
\end{equation*}
and
\[
\sum_{|\be| =1}\sum_{(x,y) \in \ZZ^2}  \Big({\tilde G} ((x,y)+\be) -{\tilde G} (x,y) \Big)^2
\le C \sum_{|\be| =1}\sum_{\substack{(x,y) \in \Delta^{\pm}\\ (x,y) + \be \in \Delta^{\pm}}}  \Big(G((x,y)+\be) -G(x,y) \Big)^2 .
\]

Thus, we have
\begin{multline*}
\Big\langle  \varphi, \left( z-{\mc S}_n \right)^{-1}\varphi \Big \rangle
\le  C_0 \sup_{G} \left\{  \sum_{(x,y) \in \ZZ^2} \Phi (x,y) G(x,y)  - C_1 (z+\gamma_n) \sum_{(x, y) \in \ZZ^2} G^2 (x,y) \right. \\
\left. \quad \quad \quad -C_2 \sum_{|\be| =1}\sum_{(x,y) \in \ZZ^2 }  \Big( G((x,y)+\be) -G(x,y) \Big)^2\right\}
\end{multline*}
where the supremum is now taken over all symmetric local functions $G:\ZZ^2 \to \RR$.  notice that the last variational formula is equal to the resolvent norm, for a simple symmetric two dimensional random walk, of the function $\Phi$.

Then by using Fourier transform, it is easy to see that (see \cite{BG}) the last supremum is equal to
\begin{equation}
\label{eq:intpsi}
\cfrac{C_0}{4} \int_{[0,1]^2} \cfrac{|{\widehat \Phi}(\bk)|^2}{C_1 [z+\gamma_n] + 4C_2 \sum_{i=1}^2 \sin^{2} (\pi k_i) } d\bk
\end{equation}
where the fourier transform ${\widehat \Phi}$ of $\Phi$ is given by
\begin{equation*}
{\widehat \Phi} ( \bk) = \sum_{(x,y) \in \ZZ^2} \Phi (x,y) e^{2i\pi (k_1 x +k_2 y)}, \quad \bk= (k_1,k_2) \in [0,1]^2.
\end{equation*}

Thus we have reduced the problem to estimate the behavior w.r.t. $n$ of the integral  \eqref{eq:intpsi} with $z= 1/ (tn^a)$. Since the constants $t, C_0, C_1,C_2$ do not play any role, we fix them equal to $1$.

The function $\varphi$ can be rewritten as
\begin{equation*}
\begin{split}
\varphi (\omega) &= \sum_{v>u} \left\{ \sum_{z \in \ZZ} (\nabla_n f) \big( \tfrac{z}{n}\big) \big[ \psi_{|v-u|} (u-z-1) -\psi_{|v-u|} (u-z)\big]\right\}\;  \omega_u\omega_v\\
&= \sum_{u,v} \Phi(u,v) \omega_u \omega_v
\end{split}
\end{equation*}
with the symmetric function $\Phi$ given by
$$\Phi (u,v) = \cfrac{{\bf 1}_{u \ne v}}{2} \; \sum_{z \in \ZZ} (\nabla_n f) \big( \tfrac{z}{n}\big) \big[ \psi_{|v-u|} (u\wedge v -z-1) -\psi_{|v-u|} (u \wedge v-z)\big].$$
Its Fourier transform is given by
\begin{equation*}
\begin{split}
{\widehat \Phi} (\bk) &= \sum_{u \in \ZZ} \sum_{j=1}^{\infty} \Phi (u,u+j) e^{2i\pi ( k_1 u + k_2 (u+j))}+\sum_{u \in \ZZ} \sum_{j=1}^{\infty} \Phi (u+j,u) e^{2i\pi ( k_1 (u+j) + k_2 u)}\\
&= - \cfrac{1}{2} \; {\cF_n(\Delta_n f)}( n(k_1+k_2))\; \sum_{j=1}^{\infty} (e^{2i\pi k_1 j} +e^{2i \pi k_2 j}) {\widehat \psi}_j (k_1 +k_2).
\end{split}
\end{equation*}
Since $|{\cF_n(\Delta_n f)} (\xi)| \le C_0$ for a suitable constant $C_0>0$ independent of $n$ and $\xi$, it follows that there exist constants $C,C'>0$ such that
\begin{equation*}
\begin{split}
| {\widehat \Phi} (\bk) |^2 &\le C \big|  \sum_{j=1}^{\infty} (e^{2i\pi k_1 j} +e^{2i \pi k_2 j}) {\widehat \psi}_j (k_1 +k_2)\big|^2\\
&=C \bigg|   \sum_{j=2}^{\infty} \Big\{ (e^{2i\pi k_1 j-2i\pi k_2} +e^{2i \pi k_2 j-2i\pi k_1} +e^{2i\pi k_1 (j-1)} +e^{2i \pi k_2 (j-1)}  ) {\widehat \rho}_j (k_1 +k_2)\Big\} \\
& \qquad \qquad + (e^{2i\pi(k_1-k_2)}+e^{2i\pi(k_2-k_1)}){\widehat \rho}_1 (k_1 +k_2)\bigg|^2\\
& \le C' \cfrac{|{\widehat \rho_1} (k_1 +k_2)|^2}{(1-|X(k_1 +k_2)|)^2}.
\end{split}
\end{equation*}

To get the last inequality, we used the explicit form of ${\widehat \rho}_j$ and the fact that \[|1-X(k_1+k_2) w| \geqslant 1-|X(k_1 +k_2)| >0\]for any complex number $w$ of modulus one. Then, from Lemma \ref{lem:estimates} it follows that
\begin{equation*}
\cfrac{|{\widehat \rho_1} (\theta)|^2}{(1-|X(\theta)|)^2} \le \cfrac{C}{\gamma_n\big[ \gamma_n + \sin^2 (\pi \theta)\big]}.
\end{equation*}

\subsection{Estimate of an integral\label{app:estimate}}

It remains to study the behavior of the integral
\begin{equation*}
I_n:=\int_{[0,1]^2} \; d\bk \left\{ \cfrac{1}{ (n^{-a} +\gamma_n) + \sin^2 (\pi k_1) + \sin^2 (\pi k_2)} \times \cfrac{1} {\gamma_n\, \big[ \gamma_n + \sin^2 (\pi(k_1 +k_2) )\big] } \right\} .
\end{equation*}

The leading term in this integral is provided when
\begin{enumerate}[(1)]
\item\label{item1} either $\bk$ is an extremal point of $[0,1]^2$, i.e. $(0,0), (0,1), (1,0), (1,1)$,
\item\label{item2} or $\bk$ belongs to the diagonal $\{k_1+k_2=1\}$, and $\bk$ is not one of the previous four points.
\end{enumerate}

We first consider the first case \eqref{item1}. By periodicity, we can assume that the extremal point is $(0,0)$. We also have that $a > b$, and then $\gamma_n \gg n^{-a}$. We perform a Taylor expansion and forget about the constants. We are reduced to compute the order as $n$ goes to $\infty$ of
\begin{equation*}
\int_{[0,1]^2}  \cfrac{d\bk} {\gamma_n \, \big[ \gamma_n + k_1^2 +k_2^2 \big] \, \big[\gamma_n + (k_1 +k_2)^2 \big]} \approx \int_0^1  \cfrac{r\, dr}{\gamma_n \big[ \gamma_n+r^2\big]^2} \approx n^{2b}.
\end{equation*}
The first estimate comes after a change into polar coordinates, and the last one is deduced from another change of variables:
\[ \int_0^1  \cfrac{r\, dr}{\gamma_n \big[ \gamma_n+r^2\big]^2} = \int_0^{1/\sqrt{\gamma_n}} \cfrac{du}{\gamma_n^2 \big[ 1+u^2 \big]^2} \approx n^{2b}.\]
In the second case \eqref{item2}, with the same argument we are reduced to investigate the behavior of
\[\int_0^1  \cfrac{r\, dr}{\gamma_n^2 \big[ \gamma_n+r^2\big]} \] which is of the same order $n^{2b}$.

When $\bk$ is not close to one of these points,  we can bound by below  $\sin^2 (\pi k_1) + \sin^2 (\pi k_2)$ and $ \sin^2 (\pi(k_1 +k_2) )$ by a strictly positive constant independent of $n$ and then show that the corresponding integral gives a smaller contribution. Finally, $I_n$ is of order $n^{2b}$.

\section{Computations in the case $b>1$}\label{appc}

\subsection{Computations involving the generator}
\label{sec:Abb}

In this subsection, we explain how to obtain the differential equations in Proposition \ref{prop:equadiff}. Let $f: \bb Z \to \bb R$ be a function of finite support, and let ${\mc E} (f) : \Omega \to \R$ be defined as
\begin{equation*}
{\mc E} (f) =\sum_{{x \in \bb Z}} f(x)\omega_x^2.
\end{equation*}
A simple computation shows that
\begin{equation*}
\mc S_n{\mc E} (f) =\sum_{{x \in \bb Z}} \Delta f(x)\omega_x^2,
\end{equation*}
where $\Delta f(x)  =  f(x+1)+f(x-1) -2f(x)$ is the discrete Laplacian on $\bb Z$.
On the other hand
\begin{equation*}
\mc A{\mc E} (f) =-2\sum_{{x \in \bb Z} } \nabla f(x)\omega_x\omega_{x+1},
\end{equation*}
where $\nabla f(x)  =  f(x+1) -f(x)$ is the discrete right-derivative in $\bb Z$. Let $f: \bb Z^2 \to \bb R$ be a symmetric function of finite support, and let ${ Q} (f) : \Omega \to \R$ be defined as

\begin{equation*}
{Q} (f) =\sum_{\substack{x,y \in \bb Z \\ x \neq y}} f(x,y) \omega_x \omega_y.
\end{equation*}
Define $\Delta f: \bb Z^2 \to \bb R$ as
\begin{equation*}
\Delta f(x,y)  =  f(x+1,y)+f(x-1,y) + f(x,y+1)+f(x,y-1)-4f(x,y)
\end{equation*}
for any $x,y \in \bb Z$ and $ A f: \bb Z^2 \to \bb R$ by
\begin{equation*} Af(x,y) = f(x-1,y) + f(x,y-1) - f(x+1,y) - f(x,y+1)
\end{equation*}
for any $x,y \in \bb Z$. Notice that $\Delta f$ is the discrete Laplacian on the lattice $\bb Z^2$ and $\mc A f$ is a possible definition of the discrete derivative of $f$ in the direction $(-2,-2)$. Notice that we are using the same symbol $\Delta$ for the one-dimensional and two-dimensional, discrete Laplacian. From the context it will be clear which operator is used. We have that
\begin{align*}
\mc S_n Q(f) &= \sum_{\substack{x,y \in \bb Z \\ x \neq y}} \big[({\Delta} f)(x,y)-4\gamma_n f(x,y)\big] \omega_x\omega_y \\
&\qquad + 2 \sum_{x \in \bb Z}  \Big\{ \big[f(x,x+1)-f(x,x)\big] +\big[f(x,x+1) -f(x+1, x+1)\big]\Big\} \omega_x \omega_{x+1}\\
		&= Q(\Delta f-2\gamma_n \text{Id})\\
		& \qquad + 2 \sum_{x \in \bb Z}  \Big\{ \big[f(x,x+1)-f(x,x)\big] +\big[f(x,x+1) -f(x+1, x+1)\big]\Big\} \omega_x \omega_{x+1}.
\end{align*}
Similarly, we have that
\[
\mc A{Q} (f) =  Q(A f) +  2 \sum_{x \in \bb Z} \Big\{ \big[ f(x-1,x) -f(x,x+1)\big] \omega_x^2- \big[ f(x,x)-f(x+1,x+1)\big]\omega_x \omega_{x+1}\Big\}.
\]
It follows that
\begin{equation}
\label{ecA.8}
\mc L_n Q(f) = Q\big((\Delta + A-4\gamma_n \text{Id})f\big) + D(f),
\end{equation}
where the diagonal term $D(f)$ is given by
\[
D(f)  =2 \sum_{x \in \bb Z} \big(\omega_x^2- {\beta}^{-1}\big) \big(f(x- 1,x) - f(x,x+1)\big)
+ 4 \sum_{x \in \bb Z}  \big( f(x,x+1) -f(x,x)\big) \omega_x\omega_{x+1}.
\]
The normalization constant ${\beta}^{-1}$ can be added for free because $f(x,x+ 1) - f(x- 1,x)$ is a mean-zero function. The diagonal term will be of capital importance, in particular the one involving $\omega_x^2$. Notice that the operators $f \mapsto Q(f)$, $f \mapsto \mc L_n Q(f)$ are continuous maps from $\ell^2(\bb Z^2)$ to $\mathbf{L}^2(\mu_\beta)$. Therefore, an approximation procedure shows that the identities above hold true for any $f \in \ell^2(\bb Z^2)$.

Recalling the definition of $\mc S_t^n$, and assuming $\beta=1$, we deduce \eqref{ec3.10}, that is
\[ \tfrac{d}{dt} \mc S_t^n= \tfrac{1}{2n} \bb E\Big[\sum_{x\in\bb Z} g\big(\tfrac{x}{n}\big)(\omega_x^2(0)-1)\times n^{3/2} \mc L_n(\mc E (f))\Big]=-2 Q_t^n(\nabla_n f \otimes \delta)+\mc S_t^n\big(\tfrac{1}{\sqrt{n}}\Delta_nf\big).
\]

\subsection{Estimates on the Poisson equation}

Let us define the functions
\begin{equation}
\label{eq:lambdagamma}
\Lambda (x,y) = 4  \left[ \sin^2 (\pi x) + \sin^2 (\pi y)\right], \quad \Omega (x,y) = 2 \left[ \sin (2\pi x) + \sin (2 \pi y)\right].
\end{equation}
%
%
%
Several times we will use the following change of variable property proved in \cite{BGJ}.

\begin{lemma}
\label{lem:cov}
Let $f: \R^2 \to {\mathbb C}$ be a $n$-periodic function in each direction of $\R^2$. Then we have
\begin{equation*}
\int\int_{\big[-\tfrac{n}{2}, \tfrac{n}{2}\big]^2} f (k,\ell) \, dk d\ell \; = \; \int \int_{\big[-\tfrac{n}{2}, \tfrac{n}{2}\big]^2} f(\xi- \ell, \ell) \, d\xi d\ell.
\end{equation*}
\end{lemma}

\subsection{Proof of Lemma \ref{l3}}
\label{sec:proof}
The proof is similar to the proof given in \cite{BGJ}. The Fourier transform of $h_n$ is given by
\begin{equation}
\label{ecB.4}
\mc F_n({h_n}) (k,\ell) = \frac{1}{2\sqrt n} \frac{ i \Omega \big(\tfrac{k}{n},\tfrac{\ell}{n}\big) \mc F_n({f})(k+\ell)}{\big( \Lambda + 4\gamma_n -i \Omega \big) \big(\tfrac{k}{n},\tfrac{\ell}{n}\big)},
\end{equation}
where $\Lambda$ and $\Gamma$ are defined in \eqref{eq:lambdagamma}.
Observe first that
\begin{equation*}
i\,\Omega \big( \tfrac{\xi -\ell}{n}, \tfrac{\ell}{n}\big) = e^{ \tfrac{2i \pi\ell}{n}} \big(1- e^{-\tfrac{2 i \pi \xi}{n}}\big) -  e^{- \tfrac{2i \pi \ell}{n}} \big(1- e^{\tfrac{2 i \pi \xi}{n}}\big)
\end{equation*}
so that
\begin{equation}
\label{eq:omegaomega}
\left\vert i\Omega \big( \tfrac{\xi -\ell}{n}, \tfrac{\ell}{n}\big)\right\vert^2 \le 4 \Big| 1- e^{\tfrac{2 i \pi \xi}{n}} \Big|^2=16 \sin^2 \big(\tfrac{\pi\xi}{n}\big).
\end{equation}
Then, by Plancherel-Parseval's relation and by using Lemma \ref{lem:cov} we have that
\begin{equation*}
\begin{split}
\| h_n \|^2_{2,n}
		&=\iint_{\big[-\tfrac{n}{2}, \tfrac{n}{2}\big]^2} | \mc F_n(h_n) (k, \ell) |^2 dk d\ell = \frac{1}{4n} \iint_{\big[-\tfrac{n}{2}, \tfrac{n}{2}\big]^2}\frac{\Omega^2\big( \tfrac{k}{n}, \tfrac{\ell}{n}\big) \, |{\mc F_n(f)} (k+\ell)|^2}{\left(\Lambda \big( \tfrac{k}{n}, \tfrac{\ell}{n}\big)+4\gamma_n\right)^2 + \Omega^2 \big( \tfrac{k}{n}, \tfrac{\ell}{n}\big)} \; dk d\ell \\
		&\le \frac{1}{n} \int_{-n/2}^{n/2} \big|1 - e^{\tfrac{2i\pi\xi}{n}}\big|^2\; \big|mc F_n(f) (\xi)\big|^2 \; \left[ \int_{-n/2}^{n/2} \frac{d\ell}{\Lambda^2 \big( \tfrac{\xi - \ell}{n}, \tfrac{\vphantom{\xi}\ell}{n}\big ) + \Omega^2 \big( \tfrac{\xi - \ell}{n}, \tfrac{\vphantom{\xi}\ell}{n}\big)} \right] \; d\xi \\
		&= 4 n \int_{-1/2}^{1/2} \sin^2 (\pi y) |\mc F_n(f) (ny) |^2 W(y) dy,
\end{split}
\end{equation*}
where for the last equality we performed the changes of variables $y=\frac{\xi}{n}$ and $x= \frac{\ell}{n}$ and forget the positive term $4\gamma_n$. The function $W$ is defined by
\begin{equation}
\label{eq:W}
W(y) =  \int_{-1/2}^{1/2}  \frac{dx}{\Lambda^2(y-x,x) + \Omega^2(y-x, x)}.
\end{equation}
It is proved in \cite{BGJ}, Lemma F.5 that $W(y) \le C |y|^{-3/2}$ on $\big[-\tfrac{1}{2}, \tfrac{1}{2}\big]$. Hence, we get, by using the second part of Lemma \ref{lem:sfp} with $p=3$ and the elementary inequality $\sin^2 (\pi y) \le (\pi y)^2$, that
\begin{equation*}
\begin{split}
&\iint_{\big[-\tfrac{n}{2}, \tfrac{n}{2}\big]^2} | \mc F_n(h_n) (k, \ell) |^2 dk d\ell \le C' n \int_{-1/2}^{1/2} \cfrac{|y|^{1/2}}{1+ (n|y|)^3} dy = O (n^{-1/2}).
\end{split}
\end{equation*}
We have proved the first part of Lemma \ref{l3}, that is \eqref{eq:l31}. We turn now to \eqref{eq:l32}.
We denote by $G_n$ the $1$-periodic function defined by
\begin{equation}
\label{eq:fF}
G_n(y)= \frac{1}{4} \int_{-1/2}^{1/2} \cfrac{\Omega^2(y-z,z) }{4\gamma_n+\Lambda (y-z,z) -i\, \Omega (y-z,z)}\, dz.
\end{equation}
As $y\to 0$, the function $G_n$ is close (in a sense defined below) to the function $G_0$ given by
\begin{equation}
\label{eq:G0}
G_0 (y)= \cfrac{1}{2}|\pi y|^{3/2} ( 1+ i \, {\rm{sgn}} (y)) .
\end{equation}
{In fact we show in Lemma \ref{lem:G0345} that there exists one constant $C>0$ such that for any $|y| \le 1/2$ and for all $n \in \bb N_0$,
\begin{equation}
\label{eq:G000}
|G_n (y) -G_0(y)| \le C \left[ \sin^2 (\pi y) + \gamma_n^2\, |\sin (\pi y)|^{-1/2} + \gamma_n\, |\sin (\pi y)|^{1/2}\right].
\end{equation}
Recall  $\mc F f$ the (continuous) Fourier transform of $f$ from \eqref{fou tranf cont}
and denote by $q:=q(f): \R \to \R$ the function defined by
\begin{equation*}
q(x) = \int_{- \infty}^{\infty} e^{-2 i \pi xy} G_0 (y) {\mc F} f (y) dy
\end{equation*}
which coincides with $-\tfrac{1}{4} {\bb L} f (x)$. Let $q_n : {\sfrac{1}{n}} \Z \to \R$ the function defined by
$q_n \big(\tfrac{x}{n}) = {{\mc D}}_n h_n \,  \big(\tfrac{x}{n}).
$
Then, the proof of \eqref{eq:l32} reduces to the following
\begin{lemma}
We have
\begin{equation*}
\lim_{n \to + \infty} \frac{1}{n} \sum_{x \in \Z} \left[ q \big(\tfrac{x}{n} \big) -q_n \big( \tfrac{x}{n}\big)\right]^2 =0.
\end{equation*}
\end{lemma}

\begin{proof}
Since $\mc F_n(h_n)$ is a symmetric function we can easily see that
\begin{equation*}
\mc F_n(q_n) (\xi) = -\frac{i}{2}  \sum_{x \in \Z} e^{ \tfrac{2i \pi \xi x}{n}} \iint_{\big[-\tfrac{n}{2}, \tfrac{n}{2}\big]^2} e^{-\tfrac{2i \pi (k+\ell) x}{n}}  \Omega \big( \tfrac{k}{n}, \tfrac{\ell}{n}\big)  \mc F_n(h_n) (k,\ell) \, dk d\ell.
\end{equation*}
We use now Lemma \ref{lem:cov} and the inverse Fourier transform relation to get
\begin{equation*}
\begin{split}
\mc F_n(q_n) (\xi) =- \cfrac{in}{2} \int_{-n/2}^{n/2} \Omega  \big( \tfrac{\xi -\ell}{n}, \tfrac{\ell}{n}\big)  \mc F_n(h_n) (\xi- \ell,\ell) \, d\ell.
\end{split}
\end{equation*}
By the explicit expression \eqref{ecB.4} of $\mc F_n(h_n)$ we obtain that
\begin{equation*}
\mc F_n({q_n}) (\xi) = \frac{\sqrt{n}}{4} \left[\int_{-n/2}^{n/2} \frac{\Omega^2  \big( \tfrac{\xi -\ell}{n}, \tfrac{\ell}{n}\big)}{4\gamma_n+\Lambda \big( \tfrac{\xi - \ell}{n}, \tfrac{\ell}{n}\big )-i\Omega \big( \tfrac{\xi - \ell}{n}, \tfrac{\ell}{n}\big)} \, d\ell \right]\; \mc F_n(f) (\xi).
\end{equation*}
Again by the inverse Fourier transform we get that
\begin{equation*}
q_n\big(\tfrac{x}{n}\big)=\int_{-n /2}^{n/2}\; e^{- \tfrac{2i\pi \xi x}{n}}  n^{3/2} G_n \big( \tfrac{\xi}{n}\big)  \mc F_n(f) (\xi)   \; d\xi.
\end{equation*}
Then we have
\begin{equation*}
\begin{split}
q \big(\tfrac{x}{n} \big) -q_n \big( \tfrac{x}{n}\big)
		= &\int_{|\xi| \geqslant n/2} \; e^{- \tfrac{2i\pi \xi x}{n}}  \; G_0 (\xi)  \;  \mc F f (\xi)   \; d\xi\\
		+& \int_{|\xi| \le n/2} \; e^{- \tfrac{2i\pi \xi x}{n}}  \; G_0 (\xi)  \; \big[  \mc F f (\xi)  - \mc F_n(f) (\xi)  \big]\; d\xi\\
              +& n^{3/2}\int_{|\xi| \le n/2} \; e^{- \tfrac{2i\pi \xi x}{n}}  (G_0- G_n) \big( \tfrac{\xi}{n}\big) \mc F_n(f) (\xi) \; d\xi.
\end{split}
\end{equation*}
Above we have used the fact that $n^{3/2} G_0\big(\tfrac{\xi}{n})=G_0(\xi)$.
Then we use the triangular inequality and Plancherel's Theorem in the two last terms of the RHS to write
\begin{equation}
\label{eq:l2limit}
\begin{split}
 \frac{1}{n} \sum_{x \in \Z} \big[ q \big(\tfrac{x}{n} \big) -q_n \big( \tfrac{x}{n}\big)\big]^2
 		&\le \frac{1}{n} \sum_{x \in \Z} \left| \int_{|\xi| \geqslant n/2} \; e^{- \tfrac{2i\pi \xi x}{n}}  \; G_0 (\xi)  \;  (\mc F f) (\xi)   \; d\xi \right|^2\\
		&\quad + \int_{|\xi| \le n/2} \left|  G_0 (\xi)  \big[  \mc F f (\xi)  - \mc F_n(f) (\xi)  \big]\right|^2 d\xi \\
		& \quad +n^3 \int_{|\xi| \le n/2} \left|  (G_0- G_n) \big( \tfrac{\xi}{n}\big) \mc F_n(f) (\xi) \right|^2  d\xi \\
		&= (I) +(II)+(III).
\end{split}
\end{equation}
We refer to \cite{BGJ} for a proof of the following fact: both  terms $(I)$ and $(II)$ give a trivial contribution in \eqref{eq:l2limit}. The first one is estimated by performing an integration by parts and using the fact that the Fourier transform ${\mc F} f$ of $f$ is in the Schwartz space and that $G_0$ and $G_0'$ grow polynomially. For the second one, a change of variables $\xi=\frac{y}{n}$ and the fact that $f$ is in the Schwartz space together with  Lemma \ref{lem:sfp} permit to conclude.

The contribution of $(III)$ is estimated by using \eqref{eq:G000} and is different from \cite{BGJ}.  Recall the trivial inequality $\vert \sin(\pi y) \vert \le \vert \pi y \vert$, for $\vert y \vert \le 1/2$. The first term of the RHS gives the upper bound
\begin{equation*}
 \frac{C}{n} \int_{|\xi| \le n/2} \; | \xi|^4 \; | \mc F_n(f) (\xi) |^2 \; d\xi = C \int_{-1/2}^{1/2}  n^4|z|^4 | \mc F_n(f) (nz) |^2 dz
\end{equation*}
which goes to $0$, as $n\to\infty$, by Lemma \ref{lem:sfp} applied with $p=2$. We now deal with the last term of the RHS of \eqref{eq:G000}, which gives the upper bound
\begin{equation*}
 C n^2 \gamma_n^2\int_{|\xi| \le n/2} \; | \xi| \; | \mc F_n(f) (\xi) |^2 \; d\xi = C\; (\gamma_n^2 n^3) \int_{-1/2}^{1/2} n|z| \; | \mc F_n(f) (nz) |^2 dz,
\end{equation*}
which goes to $0$, as $n\to\infty$. Indeed, this is a consequence of Lemma \ref{lem:sfp}, and of the fact that $\gamma_n^2 n^2$ goes to 0 (as long as $b > 1$). The second part gives
\[C n^3 \gamma_n^4 \int_{|\xi|\le n/2} \; \frac{ | \mc F_n(f) (\xi) |^2}{| \sin\big( \tfrac{\pi \xi}{n}\big)|} d\xi=Cn^4\gamma_n^4 \int_{-1/2}^{1/2} \frac{| {\mc F_n(f)} (nz) |^2}{|\sin(\pi z)|}dz.\]
Again, since $z \mapsto |\sin(\pi z)|^{-1}$ is integrable on $\big[\tfrac{-1}{2};\tfrac{1}{2} \big]$, and thanks to Lemma \ref{lem:sfp}, this bound goes to 0 as soon as $\gamma_n^4 n^4$ goes to 0, which is automatically satisfied if $b>1$.
\end{proof}

\subsection{Proof of Lemma \ref{lem:5,7.}}
\label{sec:proof2}
The proof is similar to the proof given in \cite{BGJ}.  Let $w_n: \tfrac{1}{n}\Z \to \R$ be defined by
\begin{equation*}
w_n \big(\tfrac{x}{n}\big) = h_n \big( \tfrac{x}{n} , \tfrac{x+1}{n} \big) -h_n \big( \tfrac{x}{n} , \tfrac{x}{n} \big)
\end{equation*}
and observe that
\begin{equation*}
\tfrac{1}{\sqrt{n}} {\tilde{\mc D}_n} h_n \; \big( \tfrac{x}{n}, \tfrac{y}{n} \big) = n^{3/2}
\begin{cases}
w_n \big(\tfrac{x}{n} \big), &y =x+1,\\
w_n \big(\tfrac{x-1}{n} \big),& y =x-1,\\
0, & \text{otherwise}.
\end{cases}
\end{equation*}
The Fourier transform $\mc F_n(v_n)$ is thus given by
\begin{equation}
\label{eq:vnhat}
\mc F_n(v_n) (k, \ell) = -\cfrac{1}{n} \, \cfrac{e^{2i \pi \tfrac{k}{n}} +e^{2i \pi \tfrac{\ell}{n}} }{ \big( \Lambda +4 \gamma_n -i \Omega \big) \big( \tfrac{k}{n}, \tfrac{\ell}{n}\big)} \, \mc F_n(w_n) (k+\ell).
\end{equation}
By using Lemma \ref{lem:cov}, we have that the Fourier transform of $w_n$ is given by
\begin{equation*}
\mc F_n(w_n) (\xi) =  \int_{-n/2}^{n/2}  \mc F_n(h_n) (\xi-\ell,\ell) \big\{  e^{-\tfrac{2i\pi\ell}{n}}-1 \big\} \, d\ell
\end{equation*}
By \eqref{ecB.4} we get
\begin{equation}
\label{eq:wnhat}
\mc F_n(w_n) (\xi) = -\frac{\sqrt{n}}{2} I_n\big( \tfrac{\xi}{n} \big) \mc F_n(f) (\xi)
\end{equation}
where the function $I_n$ is defined by
\begin{equation}
\label{eq:I}
I_n(y)= \int_{-1/2}^{1/2} \left( \cfrac{ \, i \, \Omega R }{\Lambda + 4\gamma_n -i \Omega} \right) \, (y-x,x) dx.
\end{equation}

\paragraph{\sc Proof of \eqref{est vn}}
As in \cite{BGJ}, we can easily get
\begin{equation*}
\| v_n \|_{2,n}^2 \le  Cn  \int_{-1/2}^{1/2} |\mc F_n(f) (ny)|^2 \big| I_n  (y) \big|^2 W(y) dy\\
\end{equation*}
by using Lemma \ref{lem:sfp}. Then, from  Lemma \ref{lem:I} and since $W(y)\le C\vert y \vert^{-3/2}$ (see \cite{BGJ}, Lemma F.5), we get that \begin{equation*}
\|v_n\|_{2,n}^2 \le Cn \int_{-1/2}^{1/2} |\mc F_n(f) (ny)|^2 |\sin (\pi y)|^{3/2} dy \le  C \int_{-1/2}^{1/2}  \cfrac{|y|^{3/2}}{1+|ny|^p} dy = \cfrac{C}{n^{3/2}} \int_{-n/2}^{n/2} \cfrac{|z|^{3/2}}{1+|z|^p} dz,
\end{equation*}
which goes to $0$, as soon as $p$ is chosen bigger than $3$.
\paragraph{\sc Proof of \eqref{est der vn}}

Following \cite{BGJ}, straightforward computations lead to

\begin{equation*}
{\mc F_n({{\mc D}_n v_n}})(\xi) =-{n}\big(1-e^{\tfrac{2i\pi \xi }{n}}\big){\mc F_n(w_n)} (\xi)J_n\big (\tfrac{\xi}{n}\big),
\end{equation*}

where  $J_n$ is given by
\begin{equation}
\label{eq:J}
J_n(y) = \int_{-1/2}^{1/2} \cfrac{1+e^{2i \pi (y-2x)}}{(\Lambda+4\gamma_n-i\Omega)(y-x,x) } dx.
\end{equation}
Now, by using \eqref{eq:wnhat} we finally get that
\begin{equation*}
{\mc F_n({{\mc D}_n v_n}})(\xi)=\frac{n^{3/2}}{2}\big( 1- e^{ \tfrac{2i \pi \xi}{n}} \big)  \mc F_n(f) (\xi) I_n \big (\tfrac{\xi}{n}\big) J_n\big (\tfrac{\xi}{n}\big),
\end{equation*}
where $I_n$ is defined by \eqref{eq:I}.

By Plancherel-Parseval's relation we have to prove that
\begin{equation*}
\begin{split}
&n^3 \int_{-n/2}^{n/2} \sin^2 \big (\pi \tfrac{\xi}{n}\big )\big| \mc F_n(f) (\xi) \big|^2 \big| I_n \big (\tfrac{\xi}{n}\big) \big|^2 \big| J_n\big (\tfrac{\xi}{n}\big) \big|^2    d\xi\\=& n^4 \int_{-1/2}^{1/2} \sin^2 (\pi y) | I_n(y)|^2 | J_n(y)|^2 |\mc F_n(f) (ny)|^2 dy
\end{split}
\end{equation*}
vanishes, as $n\to\infty$. By Lemma \ref{lem:sfp}, Lemma \ref{lem:I}, this is equivalent to show that the following term goes to 0, as $n\to\infty$:
\begin{equation*}
n^4 \int_{-1/2}^{1/2} \cfrac{|y|^4}{1+|ny|^p} dy = \cfrac{1}{n} \int_{-n/2}^{n/2} \cfrac{|z|^4}{1+|z|^p} dz.
\end{equation*}
For $p$ bigger than $5$, this term goes to $0$, as $n\to\infty$.

 \paragraph{\sc Proof of \eqref{est tilde d vn}}

 Let $\theta_n: \tfrac{1}{n}\Z \to \R$ be defined by
\[
\theta_n \big(\tfrac{x}{n}\big) = v_n \big( \tfrac{x}{n}, \tfrac{x+1}{n}\big)-v_n \big( \tfrac{x}{n}, \tfrac{x}{n}\big)
\]
and observe that
\begin{equation*}
\tfrac{1}{\sqrt{n}} {\tilde{\mc D}_n} v_n \; \big( \tfrac{x}{n}, \tfrac{y}{n} \big) = n^{3/2}
\begin{cases}
\theta_n \big(\tfrac{x}{n} \big), & y =x+1,\\
\theta_n \big(\tfrac{x-1}{n} \big), & y =x-1,\\
0, & \text{otherwise}.
\end{cases}
\end{equation*}
We have to show that
\begin{equation*}
\lim_{n \to \infty} n \sum_{x \in \Z} \theta_n^2 (x) =0
\end{equation*}
which is equivalent,  by the Plancherel-Parseval's, relation to show that
\begin{equation*}
\lim_{n \to + \infty} n^2 \int_{-n/2}^{n/2} {\mc F_n(\theta_n)} (\xi) d\xi =0.
\end{equation*}
By using Lemma \ref{lem:cov}, we have that the Fourier transform of $\theta_n$ is given by
\begin{equation*}
\mc F_n(\theta_n) (\xi) =  \int_{-n/2}^{n/2}  \mc F_n(v_n) (\xi-\ell,\ell) \big\{ e^{-\tfrac{2i\pi\ell}{n}}-1  \big\} \, d\ell.
\end{equation*}
Performing a change of variables and using \eqref{eq:vnhat} and \eqref{eq:wnhat}, we get that \[\mc F_n
(\theta_n)(\xi)=\frac{\sqrt{n}}{2}\mc F_n(f)(\xi)I_n\big(\tfrac{\xi}{n}\big) K_n\big(\tfrac{\xi}{n}\big),
\]
where $I_n$ is defined by \eqref{eq:I} and $K_n$ is given by
\begin{equation}
\label{eq:K} K_n(y)=\int_{-1/2}^{1/2} \frac{\big(e^{2i\pi (y-x)}+e^{2i\pi x}\big) \big( e^{-2i\pi x}-1 \big)}{(\Lambda+4\gamma_n-i\Omega)(y-x,x)} \, dx.\end{equation}
We need to show that $\lim_{n \to \infty} n^2 \| \theta_n \|_{2,n}^2 =0.$
By Plancherel-Parseval relation, this is equivalent to prove that
\begin{equation*}
\lim_{n \to \infty} n^3 \int_{-n/2}^{n/2} \big| \mc F_n(f) (\xi) \big|^2 \big|  I_n \big( \tfrac{\xi}{n} \big) \big|^2  \big| K_n \big( \tfrac{\xi}{n} \big) \big|^2 \, d\xi =0.
\end{equation*}
By using the change of variables $y=\xi/n$, Lemma \ref{lem:sfp}, Lemma \ref{lem:I}, we have
\[
n^3 \int_{-n/2}^{n/2} \big| \mc F_n(f) (\xi) \big|^2 \big|  I_n \big( \tfrac{\xi}{n} \big) \big|^2  \big| K_n \big( \tfrac{\xi}{n} \big) \big|^2 \, d\xi
 \le Cn^4 \int_{-1/2}^{1/2} \cfrac{|y|^4}{1+|ny|^p} dy  = \cfrac{C}{n} \int_{-n/2}^{n/2} \cfrac{|z|^4}{1+|z|^p} dz
\]
which goes to $0$, as $n\to\infty$, for $p$ bigger than $5$.

\subsection{Few integral estimates}

The following lemma is the new technical estimate which take into account the extra-term coming from the velocity-flip noise of intensity $\gamma_n$.

\begin{lemma}
\label{lem:G0345}
Recall that $G_n$ and $G_0$ are defined by \eqref{eq:fF} and \eqref{eq:G0}. There exists a constant $C>0$ such that for any $|y| \le 1/2$ and for all $n \in \bb N_0$,
\begin{equation}
\label{eq:G000}
|G_n (y) -G_0(y)| \le C \left[ \sin^2 (\pi y) + \gamma_n^2\, |\sin (\pi y)|^{-1/2} + \gamma_n\, |\sin (\pi y)|^{1/2}\right].
\end{equation}
\end{lemma}

\begin{proof}
We compute the function $G_n$ thanks to the residue theorem. For any $y \in [-1/2, 1/2]$ we denote by $w:=w(y)$ the complex number $w=e^{2i\pi y}$. By denoting $z=e^{2i \pi x}$, for $x \in [-1/2, 1/2]$, we have that
\begin{equation*}
\begin{split}
 \Lambda (y-x,x)&= 4- z (w^{-1} +1) -z^{-1}( w+1),\\
i \, \Omega (y-x,x)&= z(1-w^{-1}) + z^{-1} (w-1).
\end{split}
\end{equation*}
We denote by ${\mc C}$ the unit circle positively oriented. Then, we have
\begin{equation*}
G_n(y)= \frac{1}{16 i \pi} \oint_{\mc C} f_{w} (z) dz
\end{equation*}
where the meromorphic function $f_{w}$ is defined by
\begin{equation*}
f_w (z)= \frac{[(w-1) +z^2 (1-w^{-1})]^2}{z^2 \big(z^2 -2 z (1+\gamma_n)+w\big)}.
\end{equation*}
The poles of $f_w$ are $0$ and $z_-, z_+$ which are the two solutions of $z^2 -2z(1+\gamma_n) +w$.
We can check that
\begin{equation}
\label{eq:zpm}
z_{\pm} = (1+\gamma_n) \pm \sqrt{ \alpha}\,  e^{ i \theta/2},
\end{equation}
where
\begin{align*}
\alpha^2&=4(1+\gamma_n)^2\sin^2(\pi y)+\big[ (1+\gamma_n)^2-1 \big]^2,\\
\theta & = \arctan\left(\frac{(1+\gamma_n)^2-\cos(2\pi y)}{\sin(2\pi y)}\right)-\frac{\pi}{2}{\rm sgn}(y).
\end{align*}
Observe that $|z_-| <1$ and $|z_+|>1$. Indeed, since
\[\frac{(1+\gamma_n)^2-\cos(2\pi y)}{\sin(2\pi y)}=\frac{\gamma_n^2 + 2\gamma_n + 2 \sin^2 (\pi y)}{\sin(2\pi y)}\]
we have that $\theta \in (-\tfrac{\pi}{2},\tfrac{\pi}{2})$ (distinguish the case $y>0$ and the case $y<0$). Moreover $\gamma_n>0$ so that a simple geometric argument permits to conclude that $|z_+|>1$. Since $z_- z_+ =1$, we have $|z_-|<1$.

By the residue theorem, we have
\begin{equation*}
\oint_{\mc C} f_{w} (z) dz = 2 \pi i \big[ {\rm {Res}} (f_w, 0) +{\rm {Res}} (f_w, z_-) \big]
\end{equation*}
where ${\rm {Res}} (f_w, a)$ denotes the value of the residue of $f_w$ at pole $a$. An elementary computation shows that
\begin{equation*}
\begin{aligned}
{\rm{Res}} (f_w,0) &= \cfrac{2 (w-1)^2}{w^2}, \\
{\rm{Res}} (f_w,z_-)&= \lim_{z \to z_-} (z-z_-) f_{w} (z)= \cfrac{1}{z_- -z_+} \cfrac{\big[(w-1) + (1-w^{-1})\; z_{-}^2\big]^2}{z_-^2}.
\end{aligned}
\end{equation*}
By using the fact that $z_-^2 = 2z_-(1+\gamma_n) -w$, we obtain that
\begin{equation*}
\begin{split}
G_n (y) &= \frac{1}{8} \big[{\rm Res} (f_w, 0) +{\rm {Res}} (f_w, z_-)\big] = \frac{(w-1)^2}{4 w^2}\left[ 1 + \frac{2(1+\gamma_n)^2}{z_- -z_+} \right]\\
&= \frac{(w-1)^2}{4 w^2}\left[ 1 - \frac{(1+\gamma_n)^2}{\sqrt{\alpha}} e^{-i\theta/2}\right]\\
&= \frac{(w-1)^2}{4 w^2} - \frac{(w-1)^2}{4 w^2} \frac{(1+ \gamma_n)^{3/2}}{\sqrt{2 | \sin (\pi y)|}} \left[1+ \frac{(2+\gamma_n)^2}{4(1+\gamma_n)^2} \frac{\gamma_n^2}{\sin^2 (\pi y)}\right]^{-1/4} e^{- i \theta /2}\\
&=\frac{(w-1)^2}{4 w^2} - \frac{(w-1)^2}{4 w^2} \frac{(1+ \gamma_n)^{3/2}}{\sqrt{2 | \sin (\pi y)|}} \, e^{- i \theta /2} +\ve_n (y)
\end{split}
\end{equation*}
where
\begin{equation*}
\ve_n (y) =\frac{(w-1)^2}{4 w^2} \frac{(1+ \gamma_n)^{3/2}}{\sqrt{2 | \sin (\pi y)|}} \, e^{- i \theta /2}\, \left\{ 1-\left[1+ \frac{(2+\gamma_n)^2}{4(1+\gamma_n)^2} \frac{\gamma_n^2}{\sin^2 (\pi y)}\right]^{-1/4} \right\}.
\end{equation*}

%
We deal separately with two terms: first,
\[ \left\vert \frac{(w-1)^2}{4w^2} \right\vert = \sin^2(\pi y).\]
Second, after noticing that \[ \frac{1}{4} \le \frac{(2+\gamma_n)^2}{4(1+\gamma_n)^2} \le 1,  \] we get
\[\left\vert 1-\left[1+ \frac{(2+\gamma_n)^2}{4(1+\gamma_n)^2} \frac{\gamma_n^2}{\sin^2 (\pi y)}\right]^{-1/4} \right\vert  \le C \left\vert \left[1+ C' \frac{\gamma_n^2}{\sin^2 (\pi y)}\right]^{1/4}-1 \right\vert
\le C'' \frac{\gamma_n^2}{\sin^2(\pi y)}.
\]
Then, we get easily that
\begin{equation*}
|\ve_n (y)| \le C \gamma_n^2 \, |\sin (\pi y)|^{-1/2} .
\end{equation*}
Moreover, we have that
\begin{equation*}
|(1+ \gamma_n)^{3/2} -1| \le C \gamma_n.
\end{equation*}
Since the derivative of $\arctan$ is bounded above by a constant, we  have also that
\begin{equation*}
\begin{split}
&\big|e^{-i \theta/2} - \frac{1}{\sqrt 2} (1 + i \rm{sgn} (y))e^{-i \pi y/2}  \big|\\
&=\left| \frac{1+i \rm{sgn} (y)}{\sqrt 2}\right| \,  \left| \exp\left\{ -\frac{i}{2} \left[ \arctan \left(\tan(\pi y) + \frac{\gamma_n (2+ \gamma_n)}{\sin (2\pi y)} \right)\right] \right\} - \exp\left\{-\frac{i}{2} \tan (\pi y) \right\} \right|\\
&\le C  \left|  \arctan \left(\tan(\pi y) + \frac{\gamma_n (2+ \gamma_n)}{\sin (2\pi y)} \right)  - \tan (\pi y) \right|\le C \frac{\gamma_n}{|\sin (\pi y)|}.
\end{split}
\end{equation*}

We deduce the result
\begin{equation*}
|G_n (y) -G_0(y)| \le C \left[ \sin^2 (\pi y) + \gamma_n^2\, |\sin (\pi y)|^{-1/2} + \gamma_n\, |\sin (\pi y)|^{1/2}\right].
\end{equation*}
\end{proof}

The last Lemma below is widely inspired from \cite{BGJ}.

\begin{lemma}
\label{lem:I}
The functions $I_n$, $J_n$, and $K_n$, respectively defined by  \eqref{eq:I}, \eqref{eq:J} and \eqref{eq:K}, satisfy for any $y\in \R$ and $n \in \bb N_0$,
\begin{align*}
| I_n(y)| &\le C |\sin (\pi y)|^{3/2}\\
|J_n(y)| &\le  C |\sin (\pi y)|^{-1/2}\\
|K_n(y)| &\le  C |\sin (\pi y)|^{1/2}
\end{align*}
where $C$ is a positive constant which does not depend on $y$ nor on $n$.
\end{lemma}

\begin{proof} We divide the proof into three parts, corresponding to the three inequalities
\begin{enumerate}
\item As previously, we compute $I_n$ by using the residue theorem. For any $y \in [-\frac{1}{2}, \frac{1}{2}]$ we denote by $w:=w(y)$ the complex number $w=e^{2i\pi y}$. Then we have
\begin{equation}
\label{eq:E8}
I_n(y)= -\frac{1}{4 i \pi} \cfrac{w-1}{w} \oint_{\mc C} f_{w} (z) dz,
\end{equation}
where the meromorphic function $f_{w}$ is defined by
\begin{equation*}
f_w (z)= \frac{(z-1)(z^2 +w)}{z^2 (z-z_+)(z-z_-)}
\end{equation*}
with $z_\pm$ defined by \eqref{eq:zpm}. We recall that $|z_-| <1$ and $|z_+|>1$ so that by the residue theorem we have
$$I_n(y) = -\cfrac{w-1}{2w} \left[ {\rm{Res}} (f_w,0) +{\rm {Res}} (f_w,z_-)\right].$$
A simple computation shows that
\begin{equation*}
{\rm{Res}} (f_w,0) =1-2/w, \quad {\rm {Res}} (f_w,z_-) =1/z_-.
\end{equation*}
It follows that
\begin{equation*}
I_n(y)= - \cfrac{w-1}{2w} \left[ \frac{1}{z_-} +1 -\frac{2}{w}\right].
\end{equation*}
Replacing $w$ and $z_-$ by their explicit values we get the result.

\item Then we also compute $J_n$ by using the residue theorem:
\begin{equation*}
J_n(y)= -\frac{1}{4 i \pi} \oint_{\mc C} f_{w} (z) dz
\end{equation*}
where the meromorphic function $f_{w}$ is defined by
\begin{equation*}
f_w (z)= \frac{(z^2 +w)}{z^2 (z-z_+)(z-z_-)}
\end{equation*}
with $z_\pm$ defined by \eqref{eq:zpm}. By the residue theorem, we get
$$J_n(y)=-\frac{1}{2} \left( {\rm{Res}} (f_w, 0) + {\rm Res} (f_w, z_-) \right).$$
A simple computation shows that
\begin{equation*}
{\rm{Res}} (f_w, 0)= - \frac{w}{2}, \quad {\rm{Res}} (f_w, z_-)=\frac{2}{z_- (z_- -z_+)}.
\end{equation*}
By using the explicit expressions for $w$, $z_\pm$, we get the result.

\item Finally, it is not difficult to see that
\begin{equation*}
K_n(y)= -\cfrac{w}{w-1} I_n(y),
\end{equation*}
 and by the first estimate the result follows.
 \end{enumerate}
\end{proof}

\end{document}